\documentclass{article}
\usepackage{graphicx} 
\usepackage{microtype}

\usepackage{titling} 
\thanksmarkseries{alph}

\title{The Frenet Immersed Finite Element Method for Elliptic Interface Problems: An Error Analysis}
\author{
  Slimane Adjerid$^\ast$ \thanks{ Department of Mathematics, Virginia Tech, Blacksburg, VA 24061 USA(adjerids@vt.edu, tlin@vt.edu)\\ $\ast$ Corresponding author.} \and 
  Tao Lin\footnotemark[1] \and
   Haroun Meghaichi\thanks{Department of Mathematics, The Ohio State University, Columbus, OH 43210 USA (meghaichi.1@osu.edu)}
}

\usepackage{xcolor}

\usepackage[toc]{appendix}
\usepackage[utf8]{inputenc}

\usepackage{amsmath,amsfonts,amssymb,amsthm}
\usepackage{bbm}

\usepackage[export]{adjustbox}
 \usepackage{setspace}
	
\usepackage{bm}
\usepackage{graphicx}
\usepackage{svg}
\usepackage{pgf,tikz,pgfplots}
\usepackage{subfiles}

\usepackage{hyperref}
\usepackage{aliascnt}

\hypersetup{%
  colorlinks=true,
  linkcolor=blue,
  citecolor=blue,
  linkbordercolor=red,
  pdfborderstyle={/S/U/W 1}
}

\usepackage{cleveref}

\usepackage{caption}
\usepackage{subcaption}
\usepackage[ruled,vlined]{algorithm2e}
\usepackage{epstopdf}

\tikzset{
point/.style={circle,draw=black,inner sep=0pt,minimum size=3pt}
}
\pgfplotsset{
    soldot/.style={color=blue,only marks,mark=*}
    }

    \usepackage{csvsimple}

\pgfplotsset{compat=1.14}
\usepackage{mathrsfs}
\usetikzlibrary{arrows,calc,patterns}

\renewcommand{\[}{\begin{equation}}
\renewcommand{\]}{\end{equation}}
\renewcommand{\O}{\Omega}

\newcommand{\T}{\mathcal{T}}

\newcommand{\lessim}{\lesssim}

\newcommand{\id}{\mathrm{id}}
\newcommand{\z}{\mathbf{z}}
\newcommand{\ttau}{\boldsymbol{\tau}}

\newcommand{\hA}{\hat{A}}

\newcommand{\hgamma}{\hat{\gamma}}

\newcommand{\hp}{\hat{p}}
\newcommand{\he}{\hat{e}}
\newcommand{\hV}{\hat{\V}}
\newcommand{\hphi}{\hat{\phi}}
\newcommand{\hx}{\hat{\x}}

\newcommand{\hnabla}{\widehat{\nabla}}

\newcommand{\hGamma}{\hat{\Gamma}}

\newcommand{\hbeta}{\hat{\beta}}
\newcommand{\hK}{\hat{K}}
\renewcommand{\hbeta}{\hat{\beta}}

\newcommand{\hD}{\hat{D}}

\newcommand{\n}{\mathbf{n}}

\renewcommand{\H}{\mathcal{H}}

\newcommand{\y}{\mathbf{y}}

\newcommand{\cgamma}{\check{\gamma}}

\newcommand{\cp}{\check{p}}
\newcommand{\ce}{\check{e}}

 \newcommand{\cK}{\check{K}}
\newcommand{\cg}{\check{\g}}
\newcommand{\cn}{\check{\n}}
\newcommand{\ctau}{\check{\ttau}}
\newcommand{\cP}{\check{P}}
\newcommand{\cR}{\check{R}}
\newcommand{\cGamma}{\check{\Gamma}}

\newcommand{\cxi}{\check{\xi}}

\newcommand{\ccK}{\check{\check{K}}}

\newcommand{\tO}{\tilde{\O}}

\usepackage{stmaryrd}
\newcommand{\bb}[1]{\left\llbracket#1\right\rrbracket}

\newcommand{\cc}[1]{\left\{\!\!\left\{#1\right\}\!\!\right\}}

\newcommand{\diam}{\operatorname{diam}}

\renewcommand{\c}{\mathbf{c}}
\newcommand{\x}{\mathbf{x}}

\newcommand{\g}{\mathbf{g}}

\newcommand{\E}{\mathcal{E}}
\renewcommand{\L}{\mathscr{L}}

\newcommand{\norm}[1]{\left\lVert#1\right\rVert}

\renewcommand{\P}{\mathcal{P}}

\newcommand{\V}{\mathcal{V}}

\newcommand{\p}{\partial}

\usepackage[utf8]{inputenc}

\usepackage{cases}

\newcommand{\comment}[1]{}

\usepackage{scalerel}

\usepackage{cancel}
\newcommand{\cdelta}{
  \scaleto{
 \begin{tikzpicture}
   \draw[line width=.5pt] (0, 0) to[out=20, in=250] (.5, 0.866025) ;
  \draw[line width=.5pt] (.5, 0.866025) to[out=290, in=160] (1,0) ;
  \draw[line width=.5pt] (1,0) to[out=-180, in=0] (0, 0);
\end{tikzpicture}
}{8pt}}

\newlength{\subcolumnwidth}
\newenvironment{subcolumns}[1][0.45\columnwidth]
 {\valign\bgroup\hsize=#1\setlength{\subcolumnwidth}{\hsize}\vfil##\vfil\cr}
 {\crcr\egroup}
\newcommand{\nextsubcolumn}[1][]{%
  \cr\noalign{\hfill}
  \if\relax\detokenize{#1}\relax\else\hsize=#1\setlength{\subcolumnwidth}{\hsize}\fi
}

\newcommand{\vertiii}[1]{{\left\vert\kern-0.25ex\left\vert\kern-0.25ex\left\vert #1 
    \right\vert\kern-0.25ex\right\vert\kern-0.25ex\right\vert}}

\usepackage[
backend=biber,
citestyle=numeric,
bibstyle=numeric,
url=false,
doi=false,
giveninits=true,
maxbibnames=99,
isbn=false]{biblatex}

\DeclareSourcemap{
  \maps[datatype=bibtex, overwrite]{
    \map{
      \step[fieldset=editor, null]
    }
  }
}
\AtEveryBibitem{\clearfield{month}}
\AtEveryBibitem{\clearfield{day}}
\AtEveryBibitem{\clearfield{eprint}}
\AtEveryBibitem{\clearfield{issue}}

\DeclareFieldFormat
  [article,inbook,incollection,inproceedings,patent,thesis,unpublished]
  {title}{\emph{#1\isdot}}

  \renewbibmacro*{volume+number+eid}{%
  \printfield{volume}%
  \setunit*{\addnbspace}
  \printfield{number}%
  \setunit{\addcomma\space}%
  \printfield{eid}}
\DeclareFieldFormat[article]{number}{\mkbibparens{#1}}

\usepackage{biblatex2bibitem}

\theoremstyle{definition}

\newtheorem{lemma}{Lemma}
\newtheorem{theorem}{Theorem}

\usepackage[top=2cm, bottom=2cm,left=2cm,right=2cm]{geometry}

\newcommand{\commentout}[1]{{}} 

\begin{document}
\maketitle

\begin{abstract}
    This article presents an error analysis of the recently introduced Frenet immersed finite element (IFE) method \cite{adjeridHighOrderGeometry2024}. The Frenet IFE space    
    employed in this method is constructed to be locally conforming to the function space of the associated weak form for the interface problem. This article further establishes a critical trace inequality for the Frenet IFE functions. These features enable us to prove that the Frenet IFE method converges optimally under mesh refinement in both $L^2$ and  energy norms.
\end{abstract}

\section{Introduction}\label{sec:intro}

In this paper, we provide an error analysis of the  Frenet IFE method introduced recently in \cite{adjeridHighOrderGeometry2024} for solving the elliptic interface problem
\begin{subequations}\label{eqn:Interface_problem}
\begin{equation}
    \begin{cases}
        -\nabla\cdot (\beta\nabla u) = f,& \text{on } \Omega^+\cup\Omega^-,\\
        u_{\mid \partial \Omega} =g,
    \end{cases}
   \label{eqn:PDE}
\end{equation}
where, $\O\subset\mathbb{R}^2$ is a domain split by an interface $\Gamma$ into two subdomains $\Omega^-$ and  $\Omega^+$, and the diffusion coefficient $\beta$ is a piecewise constant function $\beta|_{\O^{\pm}}=\beta^{\pm}$ with $\beta^+\ge \beta^->0$. In addition to the BVP \eqref{eqn:PDE}, the solution is assumed to satisfy the following interface conditions across $\Gamma$:
\begin{equation}
    \bb{u}_{\Gamma}=0,\qquad \bb{\beta\p_{\n}u}_{\Gamma}=0, \label{eqn:essential_jump_conditions}
\end{equation}
\end{subequations}
where $\bb{\cdot}$ denotes the jump across $\Gamma$, that is $\bb{v}_{\Gamma}=v^{+}|_{\Gamma}-v^{-}|_{\Gamma}$ if $v|_{\O^\pm}=v^{\pm}$, and $\n$ is the normal vector on $\Gamma$ from $\Omega^-$ towards $\Omega^+$. As in \cite{adjeridHighOrderGeometry2024}, we focus on the case where $f$ has sufficient regularity such that
\begin{equation}
    \bb{\beta\p_{\n^j}\Delta u}_{\Gamma}=0,\qquad j=0,1,\dots, m-2, 
    \label{eqn:extended_jump_condition}
\end{equation}
{where $\p_{\n^j}$ stands for the $j$-th order normal derivative.}  
The elliptic interface problem is ubiquitous in science and engineering  and is used to model many physical phenomena in inhomogeneous media. Hence, it has attracted a great attention from both scientists and engineers  in the last few decades \cite{chorinNumericalSolutionNavierStokes1968,gersborg-hansenTopologyOptimizationHeat2006,guyomarchDiscontinuousGalerkinMethod2009,wangModelingElectrostaticLevitation2008}.
The main challenge with the interface problem \eqref{eqn:Interface_problem} is that its solution has a low regularity around the interface due to the discontinuity of $\beta$, 
making traditional methods based on unfitted meshes extremely slow to converge, if they converge at all. To resolve this issue interface-fitted finite element methods were introduced and analyzed \cite{babuskaFiniteElementMethod1970,brambleFiniteElementMethod1996}. These conventional finite element methods require meshes constructed according to the interface (the so-called interface-fitted meshes) but use standard finite element functions to approximate the solution. However, the use of interface-fitted meshes can be prohibitively expensive when solving problems with evolving interfaces and/or very thin layers of inhomogeneities. This drawback motivated scientists  to investigate methods based on unfitted meshes for interface problems to take advantage of the computational flexibility they offer over fitted-mesh methods. 

The IFE method is one type of finite element methods that can use interface-unfitted meshes. 
A key idea of an IFE method is to use IFE functions that are piecewise polynomials on interface elements (those elements cut by the interface) constructed according to the interface conditions, while standard finite element functions in terms of polynomials are used on non-interface elements (those elements not cut by the interface). The IFE method is attractive because it can solve interface problems at the optimal convergence rate on meshes independent of the interface. The inception of the IFE method can be traced back to \cite{liImmersedInterfaceMethod1998}, where the author considered the one-dimensional version of \eqref{eqn:Interface_problem}. Subsequently, the IFE method has been extended  to elliptic interface problems in two and three dimensions \cite{
adjeridHigherDegreeImmersed2018,
guoGroupImmersedFiniteelement2019,
guoImmersedFiniteElement2020,
guzmánFiniteElementMethod2017,
heInteriorPenaltyBilinear2010,
linPartiallyPenalizedImmersed2015a,
linRectangularImmersedFinite2001,
},
transient interface problems that involve elliptic differential operators were investigated in  \cite{
adjeridErrorEstimatesImmersed2020, 
heImmersedFiniteElement2013, linMethodLinesBased2013,
linPartiallyPenalizedImmersed2015,
linOptimalErrorBounds2020} and problems for wave propagation in inhomogeneous media \cite{
adjeridImmersedDiscontinuousGalerkin2023,
adjeridErrorEstimatesImmersed2020,
adjeridImmersedDiscontinuousGalerkin2019,
guoErrorAnalysisSymmetric2021,
linSolvingInterfaceProblems2019}.
The IFE method has also been successfully applied to interface problems associated with a system of partial differential equations, such as Stokes system \cite{
adjeridImmersedDiscontinuousFinite2015,
adjeridImmersedDiscontinuousFinite2019,
jiImmersedCRP0Element2022} as well as  the linear elasticity system
\cite{
linLockingfreeImmersedFinite2013,
linLinearBilinearImmersed2012}. Furthermore, the IFE method has been applied successfully in three spatial dimensions to solve various interface problems \cite{
guoSolvingThreedimensionalInterface2021,
han3DImmersedFinite2016,
kafafyThreedimensionalImmersedFinite2005,
vallaghéTrilinearImmersedFinite2010}, and optimally converging high-order IFE methods have been successfully constructed in 
\cite{adjeridStudyHighorderImmersed2022,
adjeridHighDegreeImmersed2017,
adjeridPthDegreeImmersed2009,
adjeridEnrichedImmersedFinite2023,
guoHigherDegreeImmersed2019}



In the recent paper \cite{adjeridHighOrderGeometry2024}, the authors developed a novel IFE method for the interface problem \eqref{eqn:Interface_problem}, called the \text{Frenet} IFE method. 
The Frenet IFE functions used with this method are constructed with the essential differential geometry described by the Frenet apparatus of the interface curve which is assumed to be nonlinear in general, an idea reminiscent to the isoparametric finite elements. They are first constructed as piecewise polynomials in the Frenet coordinates, then mapped to the interface elements by the Frenet transformation. Therefore, in contrast to IFE methods in the literature, where the IFE functions are piecewise polynomials, the Frenet IFE functions are not piecewise polynomials on the interface elements. However, Frenet IFE functions have 
two distinct features. First, they can precisely satisfy the interface conditions \eqref{eqn:essential_jump_conditions}, whereas all other IFE functions in the literature can only satisfy these conditions approximately. This is an advantageous feature since the penalty required on the interface in previous IFE methods {\color{red}\cite{adjeridEnrichedImmersedFinite2023} } is no longer necessary. Second, we note that the local construction of Frenet IFE functions is robust and does not suffer from the notorious small-cut issue \cite{adjeridHigherDegreeImmersed2018}, hence, there is no need for user chosen parameters \cite{zhuangHighDegreeDiscontinuous2019} to alleviate this issue. Furthermore, we have proved in \cite{adjeridHighOrderGeometry2024} that the Frenet IFE space has the expected optimal approximation capability regarding the polynomial space in Frenet coordinates. Numerical results given \cite{adjeridHighOrderGeometry2024} demonstrate that the standard DG scheme based this Frenet IFE space 
can solve the interface problem \eqref{eqn:Interface_problem} with an optimal convergence rate
under mesh refinement and an exponential convergence under degree refinement. 
Our goal in the current article is to theoretically establish the optimal convergence for the Frenet 
IFE method developed in \cite{adjeridHighOrderGeometry2024}.

\commentout{
     where, unlike existing IFE methods, the Frenet  IFE shape functions on a physical interface element cut by a curved interface are, in general, not piecewise polynomials. Instead, the  interface element  is transformed such that the interface becomes a line, then a space of piecewise polynomials is constructed according to the interface conditions \eqref{eqn:essential_jump_conditions}, then the piecewise polynomial basis function are mapped back to the original element to form a basis of the local IFE space. This new approach has many advantages. For instance, the interface conditions \eqref{eqn:essential_jump_conditions} are satisfied exactly by the IFE shape functions. Hence, the penalty on the interface used in many IFE methods is no longer needed. Furthermore, the local IFE construction is robust and does not suffer from the notorious small-cut issue \cite{adjeridHigherDegreeImmersed2018} without the need for user chosen parameters \cite{zhuangHighDegreeDiscontinuous2019}. In \cite{adjeridHighOrderGeometry2024}, we provide several  numerical examples showing the optimal convergence of the  Frenet IFE method under mesh refinement and exponential convergence under degree refinement, as well as an optimal $O(h^{m+1})$ approximation estimate for the Frenet IFE space. The missing rigorous error analysis in \cite{adjeridHighOrderGeometry2024} for the Frenet IFE method is addressed in this manuscript. 
} 

Since the Frenet IFE method to be analyzed in this manuscript is based on a symmetric interior penalty discontinuous Galerkin (SIPDG) weak formulation, the error analysis presented here follows the reasoning employed in the analysis of the SIPDG method \cite{arnoldInteriorPenaltyFinite1982,wheelerEllipticCollocationfiniteElement1978,riviereDiscontinuousGalerkinMethods2008} and the partially penalized IFE method \cite{adjeridHighDegreeImmersed2017,guoHigherDegreeImmersed2019,linPartiallyPenalizedImmersed2015a} where the trace inequalities play a crucial role. 
In a nutshell, the SIPDG formulation contains a stabilization term that penalizes the jump across the edges between elements, and a trace inequality is derived to ensure that for a large enough penalty parameter, the discrete bilinear form is coercive. After that, the error estimates follow from the standard arguments. In the Frenet IFE method, the shape functions on the physical interface elements are generally not piecewise polynomials, and as a result, the standard trace inequality does not hold.
To overcome this obstacle, we use the Frenet coordinates $(\eta, \xi)$ to define a locally invertible Frenet transformation that maps a rectangular interface element (cut by a curved interface) to a reference quadrilateral interface element with four curved edges cut by the vertical straight line $\eta=0$. IFE functions are piecewise polynomials in the Frenet coordinates, but the traditional trace inequality cannot be applied directly in this situation because the reference interface element has curved edges. We circumvent this issue by, first, proving a trace inequality for polynomials on the reference interface element in Frenet coordinates on a curved quadrilateral, then decomposing the IFE space in the Frenet coordinate system into a subspace of piecewise polynomials in terms of $\beta^\pm$  and a subspace of polynomials that only depends on the geometry of the interface. After proving  the trace inequality for non polynomial Frenet functions on an interface element, the error analysis proceeds in the usual manner.  


\section{Preliminaries and notation}\label{sec:notation}
Without loss of generality, we assume that $\Omega\subset \mathbb{R}^2$ is a rectangular domain, and we assume that $\Gamma\subset \O$ splits $\O$ into two subdomains $\O^+$ and $ \O^-$ such that ${\p\O^+}\cap \p\O^-=\Gamma$. Given a measurable set $\tO\subset\O$ and $s\ge 0$, we use $H^{s}(\tO)$ to denote the standard Sobolev space $W^{s,2}(\tO)$ equipped with the Sobolev norm $\norm{\cdot}_{H^{s}(\tO)}$ and the semi-norm $|\cdot|_{H^s(\tO)}$. For convenience, we will use $(\cdot,\cdot)_{\p \tO}$ and $\norm{\cdot}_{L^2(\p\tO)}$ to denote the inner product and the norm of $L^2(\p\tilde{\O})$ when applicable.  If $\tO$ intersects both $\O^+$ and $\O^-$, we let $\tO^{\pm}= \tO\cap \O^{\pm}$, and  use $\H^s(\tO,\Gamma;\beta)$ to denote the immersed Sobolev space for $s>\frac{3}{2}$
\begin{equation}
    \H^s(\tO,\Gamma;\beta)=
    \left\{v\in L^2(\tO) \mid v|_{\tO^{\pm}}\in H^{s}(\tO),\  \bb{v}_{\Gamma\cap \tO}=\bb{\beta\p_{\n}v}_{\Gamma\cap \tO}=0 \right\},
\end{equation}
equipped with the broken Sobolev norm $\norm{\cdot}_{\H^s(\tO)}^2=\norm{\cdot}_{H^s(\tO^+)}^2+\norm{\cdot}_{H^s(\tO^-)}^2$ and the standard $L^2$ norm. Additionally, we use $\norm{\cdot}$ to denote the Euclidean norm of a vector and the Frobenius norm of a matrix. 

On $\Omega$, we consider a uniform rectangular mesh $\T_h$ independent of the interface, where $h$ is the diameter of each element. However, we note here that the derivations and results in this paper extend to non-uniform quadrilateral and triangular meshes. An element  $K\in \T_h$ whose interior intersects $\Gamma$ is classified as an interface element; otherwise, $K$ is a non-interface element. 
Let $\T_h^n$ and $\T_h^i$ denote the set of all non-interface elements and the set of interface elements of $\T_h$, respectively. Additionally, we use $\E_h, ~\E_h^b$ and $\E_h^{\circ}$ to denote the set of edges, boundary edges and interior edges of $\T_h$, respectively. On an interior edge $e\in \E_h$, $\bb{u}_e$ and $\cc{u}_e$ denote the jump and the average of $u$ on $e$, respectively, and if $e\in \E^b_h$, then $\bb{u}_e$ and $\cc{u}_e$ will be $u|_e$ \cite{riviereDiscontinuousGalerkinMethods2008}. Now, we define the discrete immersed Sobolev space 
\begin{equation}
    \H^{s}(\T_h,\Gamma;\beta)= 
    \left\{v\in L^2(\O) \mid v|_{K}\in H^s(K) \text{ if } K\in \T^n_h, \text{ else } 
        v|_K\in \H^s(K,\Gamma;\beta) 
        \right\}, ~~ s>\frac{3}{2}.
        \label{eqn:discrete_sobolev}
\end{equation}
From now on, we assume that the solution $u$ of  problem  \eqref{eqn:Interface_problem} is in $\H^{s}(\O,\Gamma;\beta)$. Therefore, $u$ is the solution of the weak problem
\begin{subequations} \label{eqn:weak_form}
\begin{equation}a_h(u,v)= L_h(v),\qquad 
    \forall v\in \H^s(\T_h,\Gamma;B),
\label{eqn:a_h_L_h_form}
\end{equation}
where $a_h: [\H^s(\T_h,\Gamma;B)]^2\to \mathbb{R}$ is the symmetric bilinear form 
\begin{equation}
    a_h(u,v) = \sum_{K\in \T_h}(\beta \nabla u,\nabla v)_{K} -
    \sum_{e\in \E_h}
    \left(
     \left(\bb{\beta\p_\n u}_e,\cc{v}_e\right)_e
    +\left(\bb{\beta\p_\n v}_e,\cc{u}_e\right)_e
    -\frac{\sigma_0\gamma}{h}\left(\bb{u}_e,\bb{v}_e\right)_e\right),
    \label{eqn:a_h_expanded}
\end{equation}
and $L_h:\H^s(\T_h,\Gamma;B)\to \mathbb{R} $ is the linear form 
\begin{equation}
    L_h(v) = \sum_{K\in \T_h}\left(f,v\right)_K+\sum_{e\in\E_h^b} \left(-\beta \p_\n v+\frac{\sigma_0\gamma}{h} v,g\right)_e.
    \label{eqn:L_h_expanded}
\end{equation}
\end{subequations}
Here, $\gamma= \frac{(\beta^+)^2}{\beta^-}$ and $\sigma_0>0$ is a constant independent of the mesh size and $\beta^{\pm}$. 

In the remainder of the paper, we assume, without loss of generality, that $\Gamma$ is a Jordan curve parametrized by $\g:I=[\xi_s,\xi_e]\to \O$, hence,  $\g(\xi_e)=\g(\xi_s)$ and $\g|_{[\xi_s,\xi_e)}$ is injective. Additionally, we assume that $\g$ is a $C^{k}$ regular parametrization, where $k\ge 2$. For the construction of the Frenet transformation, we only need $\g\in C^{2}$. However, we will require that $\g \in C^{m+1}$ later in the paper to derive formulas for the Laplacian and to obtain the error estimates for polynomial approximations of degree $m+1$. Regularity in this context means that $\g'(\xi)\ne \mathbf{0}$ for any $\xi \in I$, which allows us to define the unit tangent vector $\ttau$ and the unit normal vector $\n$
\begin{equation}
    \ttau(\xi) = \norm{\g'(\xi)}^{-1} \g'(\xi),\qquad \n(\xi)=\begin{pmatrix}
        0&1\\ -1 & 0 
    \end{pmatrix}\ttau(\xi),\qquad \xi\in I. 
    \label{eqn:def_tau_n}
\end{equation}
  Let $P:\mathbb{R}\times I\to \mathbb{R}^2$ be the transformation \begin{equation}
  {(x,y) =}  P(\eta,\xi) =\g(\xi)+\eta \n(\xi),
    \label{eqn:def_P}
\end{equation}
which has two key properties that make it very useful in the construction and analysis of the proposed Frenet IFE method. 
One can check that  $P$ maps the line segment $\{0\}\times I$ into the interface $\Gamma$ and, according to the tubular neighborhood theorem \cite{abateCurvesSurfaces2012,docarmoDifferentialGeometryCurves2016,federerCurvatureMeasures1959,pressleyElementaryDifferentialGeometry2010}, also referred to as the $r$-tubular neighborhood in the literature \cite{guzmánFiniteElementMethod2017}, there exists $\epsilon>0$ such that $P|_{(-\epsilon,\epsilon)\times [\xi_s,\xi_e)}$ is injective and $N(\epsilon)=P((-\epsilon,\epsilon)\times I)$ is a tubular neighborhood of $\Gamma$ as shown in \cite{adjeridHighOrderGeometry2024}. Thus, $P$ is invertible from  
$(-\epsilon,\epsilon)\times [\xi_s,\xi_e)$ to 
$N(\epsilon)$, and we use $R = P^{-1}$ to denote its inverse. 

In the remainder of this manuscript we assume that the mesh size $h<\frac{\epsilon}{2}$, which guarantees that each interface element $K$ is contained in $N(\epsilon)$. Moreover, $K$ is contained in the fictitious element $K_F=P([-h,h]\times [\xi_0,\xi_1])$ for some $\xi_0,\xi_1$ {as shown in \autoref{fig:Frenet_transform_K_hK}}. For a given interface element $K$, we let $\hK = R(K)$ and $\hK_F = R(K_F)$, and we call them the Frenet interface element
and the Frenet fictitious element, respectively. Here, $\hK$ is a quadrilateral with curved sides contained in the rectangle $\hK_F$. We also use $\hGamma_{K_F}$ to denote the intersection of the axis $\{0\}\times \mathbb{R}$ with $\hK_F$. 

{It is worth noting here, that $\hK$ is connected in general, with the possible exception where $K$ intersects both $\g([\xi_s,\xi_s+\delta))$ and $\g([\xi_e-\delta,\xi_e))$ for some $\delta>0$. For example consider, the unit circle parametrized by $\g(\xi)=(\cos(\xi),\sin(\xi))$ on the interval $[0,2\pi)$ and consider $K$ be a small rectangle centered at $(1,0)$. However, this is not an issue since $\g$ can always be extended periodically to ensure that $\hK$ is connected. In principle, our construction and analysis does not require the use of the same parametrization on each interface element. Instead, we only need a local parametrization $\g_K$ that parametrizes the curve $\Gamma$ in $K_F$. }

Following the notation introduced in \cite{adjeridHighOrderGeometry2024}, we use $\nabla$ and $D$ to denote the gradient and the Jacobian with respect to the variables $(x,y)$, and we use $\hnabla$ and $\hD$ to denote the gradient and Jacobian with respect to the variables $(\eta,\xi)$.

\section{A summary of our results}\label{sec:space_scheme}
In this section, we describe briefly the Frenet IFE space and 
the symmetric interior penalty DG method developed in \cite{adjeridHighOrderGeometry2024}, and state the main results to be presented for the error analysis of this DG IFE method. First, let $K$ be an interface element and let $\hK^{\pm}=R(K^{\pm})$, then we define $\hV^m_{\beta}(\hK)$ to be the space of piecewise polynomials $\hphi$ such that $ \hphi|_{K^{\pm}}\in Q^m(\hK^{\pm})$ and 
\begin{equation}
         \bb{\hphi}_{\hGamma_{K_F}}=
        \bb{\hbeta\hphi_{\eta}}_{\hGamma_{K_F}}=0, ~\left(\bb{\hbeta\p_{\eta^j}\L(\hphi)}_{\hGamma_{K_F}},v\right)_{\hGamma_{K_F}} =0,\quad  \forall v\in \P^m(\hGamma_{K_F}), ~ j=0,1,\dots,m-2,
        \label{eqn:hV_conditions}
\end{equation}
where $\hbeta=\beta\circ P$, and $\L$ is a linear differential operator satisfying $\L(u\circ P)=(\Delta u)\circ P$ which was introduced in \cite{adjeridHighOrderGeometry2024}. Here, $Q^m(\hK^{\pm})$ is the usual space of tensor-product polynomials of degree not exceeding $m$ on $\hK^{\pm}$ and $\P^m(\hGamma_{K_F})$ is the space of univariable polynomials of degree not exceeding $m$ on the line segment $\hGamma_{K_F}$. Next, we define the Frenet IFE space

\begin{equation}
    \V^m_{\beta}(K) =\left\{\hphi\circ(R|_{K})\mid \hphi\in \hV^m_{\beta}(\hK)\right\}, \label{eqn:def_Vm_K}
\end{equation}
where  $\phi\in \V^m_{\beta}(K)$ iff there exists a function $\hphi\in  \hV^m_{\beta}(\hK)$ such that $\phi(\x)=\hphi(R(\x))$ for all $\x\in K$. By construction, as proved in \cite{adjeridHighOrderGeometry2024}, this space is locally conforming in the sense that $\V^m_{\beta}(K)\subset \H^{m+1}(K,\Gamma;\beta)$. Thus, every  function $\phi\in \V^m_{\beta}(K)$ satisfies the jump conditions $\bb{\phi}_{\Gamma\cap K} =\bb{\beta\p_{\n}\phi}_{\Gamma\cap K}=0$ exactly, {which is a property not shared by any other IFE method in the literature.}

To recap the Frenet IFE method developed in \cite{adjeridHighOrderGeometry2024}, we use  $\V^m_{\beta}(\T_h)$ to denote the global discontinuous Frenet IFE space 
\begin{equation}
    \V^m_{\beta}(\T_h) = \left\{\phi\mid \phi|_{K}\in Q^m(K)\text{ if } K\in \T^n_h,\ \phi|_{K}\in \V^m_{\beta}(K)\ \text{if }K\in \T^i_h\right\} .
    \label{eqn:def_global_IFE_space}
\end{equation}
Then this Frenet IFE method consists of  finding $u_h\in \V^m_{\beta}(\T_h)$ such that 
\begin{equation}
    a_h(u_h,v_h)=L_h(v_h),\qquad \forall v_h \in \V^m_{\beta}(\T_h).
    \label{eqn:compact_discrete_weak_form}
\end{equation}
Similar to the symmetric interior penalty DG method \cite{arnoldInteriorPenaltyFinite1982,wheelerEllipticCollocationfiniteElement1978} and the symmetric partially penalized IFE method \cite{guoHigherDegreeImmersed2019,linPartiallyPenalizedImmersed2015a}, the equation \eqref{eqn:compact_discrete_weak_form} leads to a symmetric linear system $\mathbf{S}\c=\mathbf{f}${, where the vector $\c$ contains  the finite element coefficients of $u_h$.} 
Hence, the error analysis for the Frenet IFE method follows the same line of reasoning as the error analysis for the SIPDG method, where the trace inequality and the approximation capability of the 
underlying finite element space are essential for the error estimation. Since we have already proved  that the Frenet IFE space has an optimal approximation capability in \cite{adjeridHighOrderGeometry2024}, our effort in the present article is to derive a {trace} inequality for the Frenet IFE functions, show that $a_h$ is coercive for a large enough penalty term, and then prove the optimal convergence of this IFE method. 

We prove a trace inequality on an interface element $K$ in two steps: In \autoref{lem:inverse_trace_hK}, we show that $\norm{p}_{L^2(\p\hK)}$ is bounded by $h^{-1/2}\norm{p}_{L^2(\hK)}$ for every $p\in Q^m(\hK)$, which is similar to the classical trace inequality even though $\hK$ is a quadrilateral with curved sides. In \autoref{thm:trace_inverse_K}, we extend this result to Frenet IFE functions in $\V^{m}_{\beta}(K)$, and obtain the following trace inequality : 
$$ \norm{\beta \nabla \phi }_{L^2(\p K)} \le C_t 
        \frac{\beta^+}{\sqrt{\beta^-}} h^{-1/2}\norm{\sqrt{\beta} \nabla \phi}_{L^2(K)},\qquad \forall
        \phi\in \V^{m}_{\beta}(K),$$
where $C_t$ is independent of the mesh and $\beta$. This trace inequality will be employed later in \autoref{sec:analysis} to show that $a_h$ is coercive, which is combined with the approximation capability result established in \cite{adjeridHighOrderGeometry2024} to yield the following error estimates stated in \autoref{thm:main_error_theorem}
$$\norm{u-u_h}_{L^2(\O)}\le C \left(\frac{\beta^+}{\beta^-}\right)^2h^{m+1} \norm{u}_{\H^{m+1}(\O)}, $$ 
and 
$$\vertiii{u-u_h}_h\le C \frac{\beta^+}{\sqrt{\beta^-}}h^m \norm{u}_{\H^{m+1}(\O)},$$
where $\vertiii{\cdot}_h$ is the energy norm defined below in \eqref{eqn:def_energy_norm}.

\section{The trace  inequality on interface elements}\label{sec:trace_inv}

In this section, we derive some useful inequalities on the Frenet IFE space $\V^m_{\beta}(K)$ that will be used later in the analysis of the scheme. A very critical one is a trace inequality for the Frenet IFE functions. Similar to the analysis of the SIPDG method \cite{riviereDiscontinuousGalerkinMethods2008} and partially penalized IFE methods 
\cite{adjeridEnrichedImmersedFinite2023, guoHigherDegreeImmersed2019}, the trace inequality plays a crucial role in proving the coercivity of the bilinear form in the Frenet IFE method \eqref{eqn:compact_discrete_weak_form}
and consequently shows up in the error estimates. However, the trace inequality for traditional finite element functions in terms of polynomials and the related analysis are not directly applicable in our analysis because the Frenet IFE functions are not piecewise polynomials. Our approach is to first establish the necessary inequalities for functions in $\hV^m_{\beta}(\hK)$, which are piecewise polynomials. Then, we prove the trace inequality for Frenet functions in $\V^m_{\beta}(K)$ by analyzing how the inequalities established in $\hV^m_{\beta}(\hK)$ are transformed through the 
nonlinear Frenet mapping $P$ defined by \eqref{eqn:def_P}. One essential hurdle in this procedure is that $\hK$ is a quadrilateral with curved edges, not a polygon. We overcome this difficulty by 
approximating $\hK$ with a {parallelogram}.    

In the remainder of this paper, for conciseness, we will use $a\lesssim b$ to denote $a \le C b$ where $C$ is constant independent of the mesh size $h$ or the relative position of the interface to the mesh, and we use $a\simeq b$ to denote $a\lesssim b$ and $b\lesssim a$. Following  \cite{adjeridHighOrderGeometry2024}, we assume that the parametrization $\g$ is regular and $C^{m+1}$ smooth. Therefore,
\begin{equation}\norm{\g'}\simeq 1,\qquad \norm{\g^{(k)}}\lesssim 1,\quad k=1,2,\dots,m+1.
\label{eqn:norm_g_assumption}
\end{equation}
Furthermore, we will assume that  $h \kappa_{\Gamma}\le\frac{1}{2}$, where  $\kappa_{\Gamma} =\max \norm{\g'}^{-3} \big| \det(\g',\g'')\big|$ is the maximum curvature {of $\Gamma$}. This last condition ensures that the Jacobian of $R$ is well-defined and bounded {since $\det(\hD P)=\norm{\g'(\xi)}(1+\eta \kappa(\xi))\gtrsim 1$.} Furthermore, we have 
\begin{equation}
    \norm{\hD P(\eta,\xi) \mathbf{v}}\simeq \norm{\mathbf{v}},\ \forall \mathbf{v}\in \mathbb{R}^2,\ \forall (\eta,\xi)\in \hK.
    \label{eqn:DP_v_equiv_v}
\end{equation}
For every interface element $K \in \T_h^i$, recall that $\hK=R(K)$, then $\hK$ is contained in the rectangle $\hK_F = [-h,h]\times [\xi_0,\xi_1]$ as illustrated in \autoref{fig:Frenet_transform_K_hK}. From the figure, we also observe that the interface ${{\Gamma}_{K_F}}=\g([\xi_0,\xi_1])$ is transformed into a vertical line segment {$\{0\}\times [\xi_0,\xi_1]$}, which simplifies the construction of high order IFE shape functions considerably compared to existing IFE methods. However, the straight edges of $K$ are transformed into the curved edges of $\hK$ on which the classical inverse and trace inequalities do not apply directly. {To overcome this hurdle, we use a linear approximation of the local interface $\Gamma_{K_F}$ to form an affine mapping which maps the interface element $K$ to a parallelogram that is a good approximation to $\hK$.} Then we will use this approximation to derive the desired inverse and trace inequalities in \autoref{subsec:Preliminary_inequalities} and \autoref{subsec:Inverse_trace}. 

\commentout{
    For this reason, we will resort to approximating the local interface $\Gamma_{K_F}$ with a line segment in \autoref{subsec:geometric_estimates}, then we will use this approximation to derive the desired inverse and trace inequalities in \autoref{subsec:Preliminary_inequalities} and \autoref{subsec:Inverse_trace}. 
}

\begin{figure}[htbp]
    \center
    \includegraphics{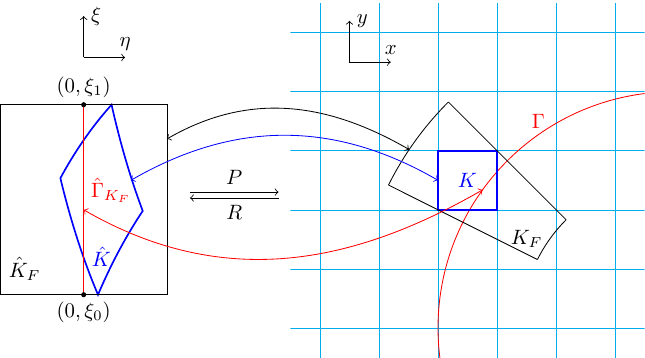}
    \caption{{An illustration of an interface element $K$ and the fictitious element $K_F$ (right), and the corresponding Frenet interface element $\hK$ and the Frenet fictitious element/rectangle $\hK_F$.}}
    \label{fig:Frenet_transform_K_hK}
\end{figure}

\subsection{Geometric estimates}\label{subsec:geometric_estimates} 

On the Frenet fictitious element $K_F=P([-h,h]\times [\xi_0,\xi_1])$, let $\cg:\mathbb{R}\to \mathbb{R}^2$ be the linear interpolant of $\g$ at $\xi_0$ and $\xi_1$ given by 

\begin{equation}
    \cg(\xi)= \frac{1}{\xi_1-\xi_0}\left((\xi-\xi_0)\g(\xi_1) + (\xi_1-\xi)\g(\xi_0)\right),\qquad \xi\in \mathbb{R}.\label{eqn:def_cg}
\end{equation}
We let  $\cGamma$  denote the line $\cg(\mathbb{R})$ with $\ctau$ and $\cn$, respectively,  denoting the unit tangential and  unit normal vectors to $\cGamma=\g(\mathbb{R})$ defined as
$$
\ctau = \frac{1}{\norm{\cg'}}\cg',\qquad \cn =\begin{pmatrix}
    0&1\\ -1 & 0 
\end{pmatrix}\ctau,
$$
which are constant vectors since $\cg$ is linear. Next, let $\cP:\mathbb{R}^2\to\mathbb{R}^2$ be the affine map $\cP(\eta,\xi)= \cg(\xi)+\eta \cn$, then it follows immediately that $\cP$ is invertible. We use $\cR$ to denote the inverse of $\cP$ and $\cK= \cR(K)$ to denote the pre-image of $K$ under $\cP$. Following \cite{adjeridHighOrderGeometry2024}, we have $\cR(K)\subset [-h,h]\times[\cxi_0,\cxi_1]$ where $\cxi_0,\cxi_1$ can be obtained from the second components of $\cR(A_i)$ for $1\le i\le 4$. 

 The advantage of the map $\cP$ is that  $\cK$ is a {parallelogram} with straight edges (unlike $\hK=R(K)$ which has curved edges). However, before attempting the standard scaling argument on the
 quadrilateral $\cK$, we need a few  estimates about the shape and size of $\cK$. First, we recall the following lemma from \cite{adjeridHighOrderGeometry2024}
\begin{lemma}\label{lem:xi1_minus_xi0}
    Let $K$ be an interface element and its associated $[\xi_0,\xi_1]$ defined above, then 
    \begin{equation}
        \xi_1-\xi_0 \simeq h,
        \label{eqn:xi1_minus_xi0}
    \end{equation}
    where $h=\diam(K)$. 
\end{lemma}

Following the reasoning used in \cite{adjeridHighOrderGeometry2024}, we can show that $\g(\xi_1)-\g(\xi_0)\simeq h$. Consequently, we have 
\begin{equation}
    \norm{\cg'}\simeq 1.
    \label{eqn:norm_cg_prime}
\end{equation}
Therefore,  $\cg$ satisfies the conditions \eqref{eqn:norm_g_assumption}, and we can apply \autoref{lem:xi1_minus_xi0} to $\cg$ to obtain

\begin{equation}
    \cxi_1-\cxi_0 \simeq h.\label{eqn:cxi1_minus_cxi0}
\end{equation}
Moreover, the equivalence \eqref{eqn:norm_cg_prime} implies that  the Jacobians of $\cP$ and $\cR$ satisfy
\begin{equation}
    \left|\det(\hat{D}\cP )\right|\simeq
    \left|\det(D\cR) \right| \simeq 1,
    \qquad 
    \norm{\hat{D}\cP \mathbf{u}}\simeq \norm{D\cR\mathbf{u}} \simeq \norm{\mathbf{u}},\ \forall\mathbf{u}\in \mathbb{R}^2,
\label{eqn:jacobian_cP_cR_good}
\end{equation}
where $\norm{\cdot}$ is the Euclidean norm on $\mathbb{R}^2$.

Finally, we will use $T:\hK\to \cK$ to denote the map $T(\hx)=\cR(P(\hx))$ {for $\hx=(\eta,\xi)\in\hK$}, and we refer to it as the transition map. It follows from  assumptions \eqref{eqn:norm_g_assumption} and \eqref{eqn:jacobian_cP_cR_good} that $T$ is continuously differentiable and 
\begin{equation}\left|\det(\hD T)\right|\simeq 1.\label{eqn:jacobian_T_good}\end{equation}

The map $T$ will be crucial later in this section, since it relates {quantities} on the {parallelogram} $\cK$ to {quantities} on the curved Frenet element $\hK$. In particular,  \eqref{eqn:jacobian_T_good} ensures that the norm of a function $f$ on $\cK$ is comparable to the norm of $f\circ T$ on $\hK$.


\subsection{Preliminary inequalities}\label{subsec:Preliminary_inequalities} 

Using \eqref{eqn:jacobian_cP_cR_good} and the classical scaling argument \cite{ciarletFiniteElementMethod2002}, we can derive the following inverse and trace inequalities on $\cK$

\begin{equation}
    \norm{\nabla p}_{L^2(\cK)} \lesssim h^{-1} \norm{p}_{L^2(\cK)},\qquad \norm{ p}_{L^2(\p\cK)} \lesssim h^{-1/2} \norm{p}_{L^2(\cK)},\qquad \forall p \in Q^m(\cK).
    \label{eqn:trace_inverse_cK}
\end{equation}






{Our goal in the remainder of this subsection is to extend inequalities \eqref{eqn:trace_inverse_cK} to  polynomials on $\hK$} which is achieved by investigating the transition map $T$ further. First, we combine the classical linear interpolation error bound with the assumption $\norm{\g''}\lesssim 1$ and \eqref{eqn:xi1_minus_xi0} to obtain
\begin{equation}
    \norm{\cg-\g}_{C^0([\xi_0,\xi_1])}\lesssim h^2,
    \qquad \norm{\cg'-\g'}_{C^0([\xi_0,\xi_1])}\lesssim h.\label{eqn:g_minus_cg}
\end{equation}
Next, we use \eqref{eqn:g_minus_cg} to derive the following inequalities 

\begin{equation}
   \norm{\ctau-\ttau}_{C^0([\xi_0,\xi_1])}\lesssim h,\qquad \norm{\cn-\n}_{C^0([\xi_0,\xi_1])}\lesssim h.
    \label{eqn:ctau_minus_tau}
\end{equation}

{From the inequalities \eqref{eqn:g_minus_cg} and \eqref{eqn:ctau_minus_tau}, we observe that the affine map $\cP$ is an $O(h^2)$ approximation of the nonlinear map $P$. Furthermore, the transition map $T$ is close to the identity map as the following lemma shows.

\begin{lemma}\label{lem:Tx_minus_x}
    For $T=\cR\circ P:\hK\to \cK$, we have
    \begin{equation}
        \norm{T-\id}_{C^0(\hK)}\lesssim h^2,
    \label{eqn:norm_T_minus_id}
    \end{equation}
    and \begin{equation}
        \norm{\hD T(\hx) -I_2}\lesssim h,\qquad \left|\det(\hD T(\hx)) -1\right|\lesssim h,\qquad \forall \hx = (\eta, \xi) \in \hK,
    \label{eqn:DT_minus_I}
    \end{equation}
    where $\id$ is the identity map and $I_2$ is the $2\times 2$ identity matrix.
\end{lemma}
\begin{proof}   
In the first part of the proof, our goal is to show that $\norm{T(\hx)-\hx}\lesssim h^2$, where $\hx=(\eta,\xi)\in \hK$. First, we observe that  for every $(\eta,\xi) \in \hK$, $|\eta|\le h$ since $\hK\subset \hK_{F}$. Now, we have  
    \begin{align}\norm{P(\hx)-\cP(\hx)}
    &\le \norm{\g(\xi)-\cg(\xi)} +\left|\eta\right| \norm{\cn -\n(\xi)}\notag\\ 
    &\le \norm{\g(\xi)-\cg(\xi)} +h \norm{\cn -\n(\xi)}\notag \\ 
    &\lesssim h^2, \label{eqn:Phx_minus_Phx}
    \end{align}
 where the last inequality follows from \eqref{eqn:g_minus_cg} and \eqref{eqn:ctau_minus_tau}. Since $\cR$ is an affine map, using \eqref{eqn:Phx_minus_Phx} and \eqref{eqn:jacobian_cP_cR_good}, we have 
 \begin{align*}\norm{T(\hx)-\hx}=\norm{\cR\left(P(\hx)\right)-\cR\left(\cP(\hx)\right)}&=\norm{D \cR\left(P(\hx)-\cP(\hx)\right)}\\ 
&\le \norm{D \cR}\norm{\left(P(\hx)-\cP(\hx)\right)}\\ 
&\lesssim h^2. 
 \end{align*}
Then \eqref{eqn:norm_T_minus_id} follows from taking the supremum over all $\hx\in \hK$ in the estimate above. Next, we move to proving \eqref{eqn:DT_minus_I}.  A direct calculation shows that 
   $$\hD\cP = \begin{pmatrix} \cn,\  \cg'\end{pmatrix} ,\qquad \hD P(\eta,\xi)=\left(\n(\xi), (1+\eta \kappa(\xi))\g'(\xi)\right)$$ 
   which, combined with \eqref{eqn:jacobian_cP_cR_good}, \eqref{eqn:ctau_minus_tau} and $|\eta|\le h$, lead to the first estimate in \eqref{eqn:DT_minus_I} as follows
\begin{align*}
\norm{  \hD T(\eta,\xi)-I_2} 
&=\norm{  (D \cR) \hD P(\eta,\xi)-(D \cR)(\hD \cP)} 
\lessim \norm{ \hD P(\eta,\xi)-\hD \cP},\notag \\ 
&\lessim \norm{\n(\xi)-\cn} +\norm{ (1+\eta \kappa(\xi))\g'(\xi) - \cg'}\notag\\ 
&\lessim \norm{\n(\xi)-\cn} +\norm{\g'(\xi)-\cg'} + h\norm{\kappa(\xi)\g'(\xi)} 
\notag \\ 
&\lesssim  h. 
\end{align*}
The second estimate in \eqref{eqn:DT_minus_I} follows from the first one since 
\[
\left|\det(\hD T(\hx)) -1\right|\lessim \norm{\hD T(\hx)-I_2}\lessim h. \notag 
\]
\end{proof}



}


The estimate \eqref{eqn:norm_T_minus_id} above shows that $\cK$ approximates $\hK$ well as $h\to 0$. We illustrate this in \autoref{fig:cK_getting_closer_to_hK}, where we consider the case in which  $\Gamma$ is the unit circle and $K(n)=\frac{1}{\sqrt{2}}+[-2^{-n},2^{-n}]^2$ for $n\in\{2,3,4,5\}$ and we let $\cK(n)$ and $\hK(n)$ be $\cR(K(n))$ and $R(K(n))$, respectively. We observe that as $n$ grows, the sets $\cK(n)$ and $\hK$ become closer to each other.

\begin{figure}[ht]
    \begin{subcolumns}[.5\textwidth]
    \begin{subfigure}{.5\textwidth}
        \center
        \includegraphics[width=.5\textwidth]{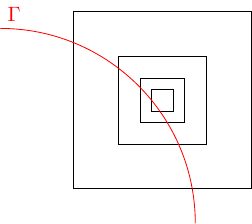}
        \caption{The sequence of elements $K(n)$}
        \label{subfig:K_getting_smaller}
    \end{subfigure}
    \nextsubcolumn   
     \begin{subfigure}{.15\textwidth}
        \center
        \includegraphics[width=.8\textwidth]{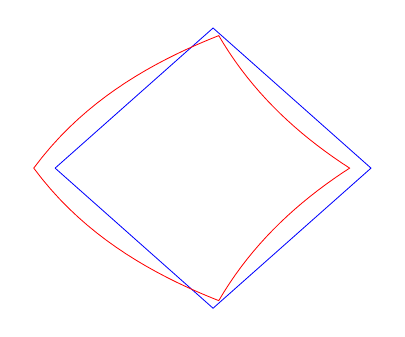}
        \caption{$n=2$}
    \end{subfigure}\begin{subfigure}{.15\textwidth}
        \center
        \includegraphics[width=.8\textwidth]{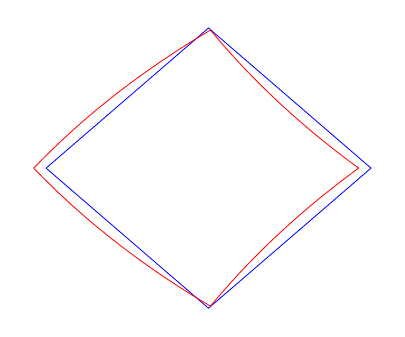}
        \caption{$n=3$}
    \end{subfigure} 
    \begin{subfigure}{.15\textwidth}
        \center
        \includegraphics[width=.8\textwidth]{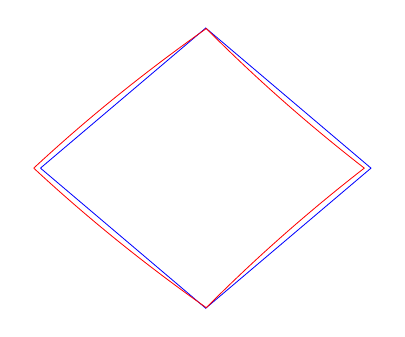}
        \caption{$n=4$}
    \end{subfigure}\begin{subfigure}{.15\textwidth}
        \center
        \includegraphics[width=.8\textwidth]{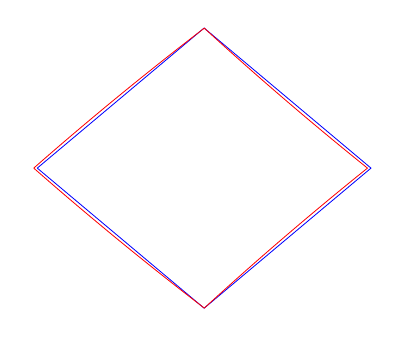}
        \caption{$n=5$}
    \end{subfigure}
\end{subcolumns}
\caption{Element $K(n)$, $n=2,3,4,5$ (left) with different sizes. The sets $\hK(n)$ (in red) and $\cK(n)$ (in blue) for $n=2,3,4,5$ (right). }
\label{fig:cK_getting_closer_to_hK}

\end{figure}

So far, we have shown that $T(\hx)$ is close to $\hx$. Naturally, we would expect $f(T(\hx))$ to be close to $f(\hx)$ for a smooth enough function $f$. However, when estimating $f(T(\hx))-f(\hx)$ using the mean value theorem, we need to 
bound the gradient $\nabla f$ on a line segment connecting $\hx$ and $T(\hx)$. This line segment might extend beyond $\cK\cup\hK$, so we need to consider a larger domain $\ccK$ containing all lines  $\overline{\hx T(\hx)}$ for all $\hx\in \hK$.  There are many ways to choose $\ccK$ but in order to simplify our calculations, we choose $\ccK$ to be the homothetic image of $\cK$ with a scaling factor of $2$, \emph{i.e.}, for the parallelogram $\cK$ with vertices  $\{\check{A}_i\}_{i=1}^{4}$  and center $\check{G}$,  $\ccK$ is the {parallelogram} containing $\cK$ with vertices $\{\check{\check{A}}_i\}_{i=1}^{4}$  such that  $\overrightarrow{\check{G}\check{\check{A}}_i}=2\overrightarrow{\check{G}\check{A}_i}$ for $1\le i \le 4$, see
\autoref{fig:illustration_of_ccK} for an illustration.

 


\begin{figure}[ht]
    \center
    \includegraphics[scale=.1]{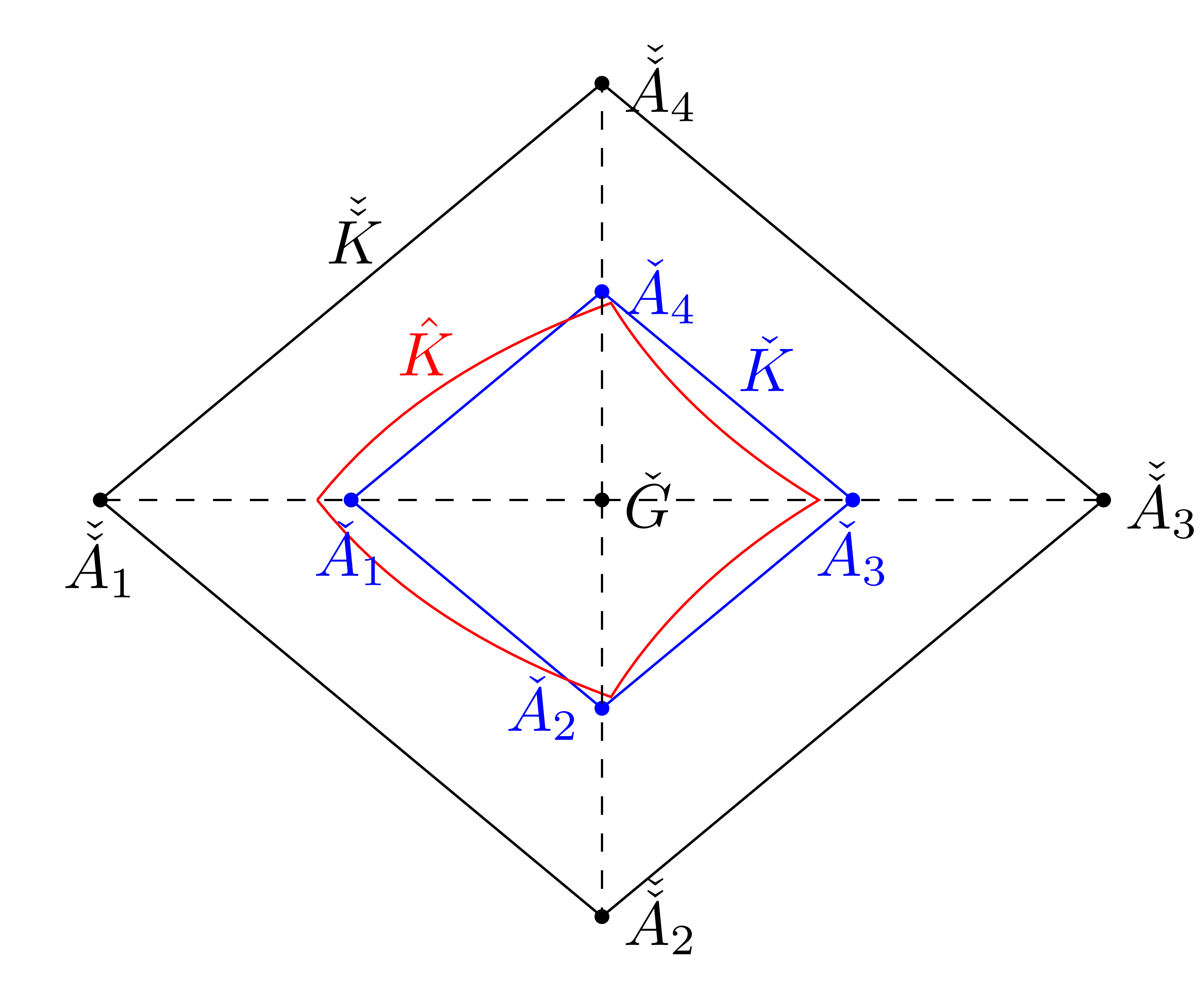}
    \caption{An illustration of a reference Frenet element $\hK$ (red),  a Frenet parallelogram $\cK$ (blue) and a fictitious Frenet parallelogram $\ccK$ (black).}
    \label{fig:illustration_of_ccK}
\end{figure}

\vfill

 Since  $\ccK$ is a homothetic image of $\cK$ with a scaling factor of $2$, the norm equivalence property \cite{xiaoHighorderExtendedFinite2020} states that the norm of a polynomial on $\cK$ is equivalent to its norm on $\ccK$, i.e., 
\begin{equation}
    \norm{p}_{L^2(\cK)}\simeq \norm{p}_{L^2(\ccK)},\qquad \forall p\in Q^m(\mathbb{R}^2).
    \label{eqn:homothetic_norm_invariance}
\end{equation}
Next, we show that $\hK$ is also a subset of $\ccK$ provided that $h$ is small enough which will allow us to derive a norm equivalence similar  to \eqref{eqn:homothetic_norm_invariance} for $\hK$ instead of $\cK$.



\begin{lemma}\label{lem:hK_subset_ccK}
{Let $K$ be an interface element and let $\ccK$ be the homothetic image of $\cK=\cR(K)$ with a scaling factor of $2$. Then, there exists $h_0>0$ such that $\hK\subset\ccK$ whenever $h=\diam(K)<h_0$.  }
\end{lemma}

\begin{proof}
    Let $\overline{B}(\mathbf{0},r)$ be the closed disk centered at the origin with radius $r$ and let $C_1$ be the hidden constant in \eqref{eqn:norm_T_minus_id}, then $\hK\subset \cK+\overline{B}(\mathbf{0},C_1h^2)$. 
     Next, let $\check{G}$ be the center of $\cK$ as shown in \autoref{fig:illustration_of_ccK}, then $G=\cP(\check{G})$ is the center of $K$. Therefore, 
    \begin{equation}
        \min_{\check{\x}\in \p \cK}\norm{\check{G}-\check{\x}}= \min_{\x\in \p K} \norm{\cR(G)-\cR(\x)}
        =  \min_{\x\in \p K} \norm{D\cR(G-\x)}\gtrsim  \min_{\x\in \p K} \norm{G-\x}\gtrsim h,
        \label{eqn:boundary_cK_far_from_G}
    \end{equation}
    where the last two inequalities follow from \eqref{eqn:jacobian_cP_cR_good} and  the shape regularity of $\T_h$, respectively. By the construction of $\ccK$, the distance between an edge $\check{e}\in \p \cK$ and its homothetic image is the same as the distance between $\check{e}$ and $\check{G}$. Therefore, by \eqref{eqn:boundary_cK_far_from_G}
    \begin{equation}
        \norm{\check{\x}-\check{\y}} \gtrsim h,\qquad \forall \check{\x}\in \p\cK,\ \forall \check{\y}\in \p \ccK. \notag
    \end{equation}
    Consequently, there exists $C_2>0$ independent of $h$ and the relative position of the interface, such that $\cK+\overline{B}(\mathbf{0},C_2 h)\subset \ccK$. Hence, if $h<h_0=\frac{C_2}{C_1}$, we have 
    $$
    \hK\subset\cK +\overline{B}(\mathbf{0},C_1 h^2)\subset
    \cK +\overline{B}(\mathbf{0},C_2 h)
    \subset \ccK, 
    $$
    which proves this lemma. 
    

    
      

\end{proof}

\begin{lemma}\label{lem:f_minus_f_circ_T}
    Let $h_0$ be the constant from \autoref{lem:hK_subset_ccK}, then 

   \begin{equation}
        \norm{f-f\circ T}_{C^0(\hK)} \lesssim h^2\norm{\nabla f}_{C^0(\ccK)},\qquad \forall f\in C^1(\ccK),
    \label{eqn:f_minus_f_circ_T}
    \end{equation}
    for all interface elements $K$ with diameter $h<h_0$. 
\end{lemma}

\begin{proof}
    Since $\hK\subset \ccK$ and $\ccK$ is convex, we can apply the mean value theorem to obtain
    $$
        \norm{f(\x)-f(T(\x))}\le \norm{\nabla f}_{C^0(\ccK)} \norm{T(\x)-\x},\qquad \forall \x \in \hK. 
    $$
    To finish the proof, we recall that from \autoref{lem:Tx_minus_x},  $\norm{T(\x)-\x}\lessim h^2$, and take the supremum over all $\x\in \hK$. 
\end{proof}
The lemma above allows to show that the norm of a polynomial on $\hK$ is comparable to its norm on $\cK$ as stated in the following lemma. 

\begin{lemma}\label{lem:norm_cK_norm_hK}
    There {exists} $h_1>0$ such that 
    \begin{equation}
        \norm{p}_{L^2(\hK)}\simeq \norm{p}_{L^2(\cK)} ,\qquad \forall p\in Q^m(\mathbb{R}^2), 
        \label{eqn:norm_cK_norm_hK}
    \end{equation}
    for all interface elements $K$ with diameter $h<h_1$. 
\end{lemma}

\begin{proof}
    Let us assume that $h<h_0$, then by \autoref{lem:hK_subset_ccK} and \eqref{eqn:homothetic_norm_invariance}, we have 
    \begin{equation}
    \norm{p}_{L^2(\hK)}\le \norm{p}_{L^2(\ccK)}\lesssim \norm{p}_{L^2(\cK)}.
    \label{eqn:norm_hK_less_norm_cK}
    \end{equation}
On the other hand, by \eqref{eqn:jacobian_T_good}, we have 
$$\norm{p}_{L^2(\cK)}= \norm{(p\circ T)\sqrt{|\det(\hD T)|}}_{L^2(\hK)}\lesssim \norm{p\circ T}_{L^2(\hK)}.$$ 
By applying the triangle inequality and \eqref{eqn:f_minus_f_circ_T} to the right-hand side, we further have
\begin{align}\norm{p}_{L^2(\cK)}&\lesssim \norm{p}_{L^2(\hK)}
+\norm{p-p\circ T}_{L^2(\hK)}\notag\\ 
&\lesssim \norm{p}_{L^2(\hK)}
+h \norm{p-p\circ T}_{C^0(\hK)}\notag\\ 
& \lesssim \norm{p}_{L^2(\hK)}
+h^3 \norm{\nabla p}_{C^0(\ccK)}. \label{eqn:step_normp_h3}
\end{align}
Next, we use the inverse inequality for polynomials \cite{brennerMathematicalTheoryFinite2002}, the norm equivalence \eqref{eqn:homothetic_norm_invariance} and \eqref{eqn:jacobian_cP_cR_good}  to obtain 
$$\norm{\nabla p}_{C^0(\ccK)}\lesssim h^{-1}\norm{\nabla p}_{L^2(\ccK)}\lesssim h^{-1}\norm{\nabla p}_{L^2(\cK)}\lesssim h^{-2}\norm{p}_{L^2(\cK)}. $$
This, in turn,  is combined with  \eqref{eqn:step_normp_h3} to yield 
\begin{equation}\norm{p}_{L^2(\cK)} \le C\left(\norm{p}_{L^2(\hK)}
+h \norm{ p}_{L^2(\cK)}\right),
\label{eqn:C0ccK_less_L2_cK}
\end{equation}
for some constant $C$ independent of $h$ and the relative position of the interface in $K$, which, for $h< h^*= \frac{1}{2C}$, leads to $\norm{p}_{L^2(\cK)} \lesssim \norm{p}_{L^2(\hK)}$. We complete the proof  by taking  $h_1=\min(h_0,h^*)$.  
\end{proof}

Next, we derive an additional estimate on the norm of a polynomial on the boundaries of $\hK$ and $\cK$. Note that we do not expect $\norm{p}_{L^2(\p \hK)} \simeq  \norm{p}_{L^2(\p \cK)} $ since a polynomial might vanish on $\p\cK$ but not on $\p\hK$ or vice versa. Instead, we can bound each semi-norm by the norm on the whole domain $\hK$ or $\cK$ as stated in the next lemma.

\begin{lemma}
    \label{lem:norm_p_boundary_hK} 
    There exists $h_2>0$  such that  
    \begin{equation}\norm{p}_{L^2(\p \hK)} + \norm{p}_{L^2(\p \cK)} \lesssim h^{-1/2}\norm{p}_{L^2(\hK)} \simeq h^{-1/2}
    \norm{p}_{L^2(\cK)},\qquad {\forall p\in Q^m(\mathbb{R}^2)}\label{eqn:norm_p_boundary_hK}
\end{equation}
for all interface elements $K$ with diameter $h<h_2$. 
\end{lemma}

\begin{proof}
    By \autoref{lem:norm_cK_norm_hK} and \eqref{eqn:trace_inverse_cK}, we have 
   \begin{equation}
    \norm{p}_{L^2(\p \cK)} \lessim h^{-1/2}\norm{p}_{L^2(\cK)}\simeq  h^{-1/2}\norm{p}_{L^2(\hK)}.\label{eqn:norm_p_boundary_cK}
   \end{equation}

    To prove the  inequality $\norm{p}_{L^2(\p \hK)} \lesssim { h^{-1/2}}\norm{p}_{L^2(\hK)}$, we first consider an edge $e$ of $K$ and map it {using} $R$ and $\cR$ to obtain $\he={R(e)}\subset \p\hK$ and  $\ce=\cR(e) \subset \p \cK$, respectively. Let $\hgamma :I_{\he}\subset \mathbb{R}\to \he$ be an arclength parametrization of $\he$,
    and we use the transition map $T$ to  define $\cgamma=T\circ\hgamma$, which is a  parametrization of $\ce$ {as illustrated in \autoref{fig:hgamma_and_cgamma}}. This parametrization is not necessarily an arclength parametrization; however, we can use \eqref{eqn:DT_minus_I} to show that it converges to an arclength parametrization under mesh refinement:  
\begin{figure}
    \center
    \includegraphics[scale=.8]{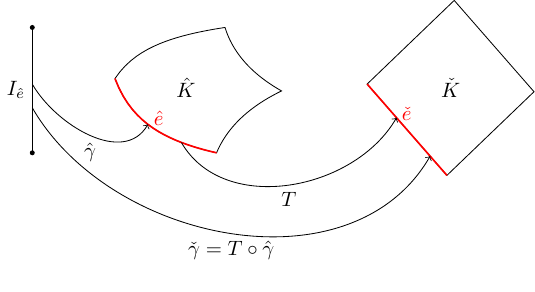}
    \caption{An illustration of the parametrizations $\hgamma$ and $\cgamma$. }
    \label{fig:hgamma_and_cgamma}
\end{figure}
    \begin{align}
        \big|\norm{\cgamma'(t)}-1\big|&\le \norm{\cgamma'(t)- \hgamma'(t)}\notag \\
       & =\norm{\hD T(\hgamma(t))\hgamma'(t)- \hgamma'(t)}\notag \\ 
        & \lesssim \max_{\hx\in\hK}\norm{\hD T(\hx)-I_2} \lesssim h \qquad \forall t\in I_{\he}. \label{eqn:norm_cgamma_hgamma} 
    \end{align}
     Next, let $\hp=p\circ \hgamma$ and $\cp=p\circ\cgamma$ for $p\in Q^m(\mathbb{R}^2)$, by the construction of $\hgamma$ and $\cgamma$, we have 
\begin{equation} \norm{p}_{L^2(\he)} = \norm{\hp }_{L^2(I_{\he})}\qquad \text{and}\qquad \norm{p}_{L^2(\ce)} = \norm{\cp \sqrt{\norm{\cgamma'}}}_{L^2(I_{\he})}.
\label{eqn:norm_p_e_ce}
\end{equation}
Therefore, by \eqref{eqn:norm_p_e_ce} and the triangle inequality, we obtain


\begin{align}\label{eqn:norm_p_he_pre}    
\begin{split}
\norm{p}_{L^2(\he)}&
        = \norm{ \hp}_{L^2(I_{\he})} = \norm{ \cp \sqrt{\norm{\cgamma'}} + \hp - \cp \sqrt{\norm{\cgamma'}} }_{L^2(I_{\he}) } \\ 
        &\le \norm{p}_{L^2(\ce)} + \norm{\hp - \cp \sqrt{\norm{\cgamma'}}}_{{L^2(I_{\he})}} \\
        & \le \norm{p}_{L^2(\ce)} + \norm{\hp - \cp }_{L^2(I_{\he})} + \norm{\left(\sqrt{\norm{\cgamma'}}-1\right){\cp} }_{L^2(I_{\he})} \\ 
        &\lesssim \norm{p}_{L^2(\ce)} + \norm{\hp - \cp }_{L^2(I_{\he})} +{h} \norm{{\cp} }_{L^2(I_{\he})}, 
\end{split}  
\end{align}
    where the {inequality $|\sqrt{a}-1|\lesssim |a-1|$} and \eqref{eqn:norm_cgamma_hgamma} are used in the last line. Next, we use \autoref{lem:f_minus_f_circ_T} to bound $\hp-\cp$:
   \begin{align}\norm{\hp - \cp }_{L^2(I_{\he})}&=
    \norm{p\circ\hgamma - p\circ\cgamma }_{L^2(I_{\he})}= \norm{p\circ\hgamma - p\circ T\circ\hgamma }_{L^2(I_{\he})}
    \notag \\  
    &\le |I_{\he}| \norm{p\circ\hgamma - p\circ T\circ\hgamma }_{C^0(I_{\he})} \le |I_{\he}| \norm{p - p\circ T }_{C^0(\hK)}
    \notag \\
    &\lesssim |I_{\he}|h^2 \norm{\nabla p}_{C^0(\ccK)}. \label{eqn:hp_minus_cp}\end{align}
 In the next step, we estimate $|I_{\he}|$ by first noticing that $P\circ \hgamma$ is a parametrization of $e$. Therefore, by \eqref{eqn:DP_v_equiv_v} and since  $\hgamma$ is an arclength parametrization, we have 
 \begin{equation}h\simeq |e| = \int_{I_{\he}} \norm{\frac{d}{dt}P(\hgamma(t))}\ dt= \int_{I_{\he}} \norm{\hD P(\hgamma(t))\hgamma'(t)}\ dt\simeq 
 \int_{I_{\he}} \norm{\hgamma'(t)}\ dt =|I_{\he}|.
 \label{eqn:size_of_I_e}
 \end{equation}
 Now, \eqref{eqn:hp_minus_cp} reads 
    \begin{equation}\norm{\hp - \cp }_{L^2(I_{\he})} \lesssim
                    h^3 \norm{\nabla p}_{C^0(\ccK)} 
            \lessim  h^2 \norm{\nabla p}_{L^2(\ccK)} 
            \lessim  h^2 \norm{\nabla p}_{L^2(\cK)}         
            \lesssim h \norm{p}_{L^2(\hK)},
                    \label{eqn:norm_hp_minus_cp}
    \end{equation}
where the last inequality follows from \eqref{eqn:trace_inverse_cK} and \autoref{lem:norm_cK_norm_hK}. 
Also, from \eqref{eqn:norm_cgamma_hgamma} and \eqref{eqn:norm_p_boundary_cK}, we have $\norm{\cp}_{L^2(I_{\he})}\simeq \norm{p}_{L^2(\ce)}\lesssim {h^{-1/2}}\norm{p}_{L^2(\hK)}$ for $h<h_2$ for some $h_2\in(0,h_1)$. Applying this estimate, \eqref{eqn:norm_hp_minus_cp} and 
\eqref{eqn:norm_p_boundary_cK} to \eqref{eqn:norm_p_he_pre} leads to 
\begin{align*}
    \norm{p}_{L^2(\he)} &\lessim \left({h^{-1/2}}\norm{p}_{L^2(\hK)}+{h}\norm{p}_{L^2(\hK)}+h^{1/2}\norm{p}_{L^2(\hK)}\right).
\end{align*}
Now, we take the sum over the four edges {of $\hK$}  to obtain 
$$
\norm{p}_{L^2(\p\hK)} \lessim {h^{-1/2}}\norm{p}_{L^2(\hK)},
$$
which, when combined with \eqref{eqn:norm_p_boundary_cK}, yields \eqref{eqn:norm_p_boundary_hK}.
\end{proof}

We are now ready to state our first inverse and trace inequalities on $\hK$ which is 
a quadrilateral with curved edges. Here, we are only concerned with the polynomial space $Q^m(\hK)$, which is a simpler space than $\hV^m_{\beta}(\hK)$.

\begin{lemma}\label{lem:inverse_trace_hK}
    Let $h_2$ be the constant from \autoref{lem:norm_p_boundary_hK}, then  

    \begin{equation}\norm{\nabla p}_{L^2(\hK)}+h^{-1/2}\norm{p}_{L^2(\p\hK)}\lessim h^{-1}\norm{p}_{L^2(\hK)}, \qquad \forall p\in Q^m(\hK),
\label{eqn:trace_inverse_Q_hK}
    \end{equation}
    for all interface elements $K$ with diameter $h<h_2$.
\end{lemma}

\begin{proof}
     By \autoref{lem:norm_cK_norm_hK} and \eqref{eqn:trace_inverse_cK}, we have 
     $$\norm{\nabla p}_{L^2(\hK)}\lesssim \norm{\nabla p}_{L^2(\cK)}
     \lesssim h^{-1}\norm{ p}_{L^2(\cK)}
     \lesssim h^{-1}\norm{p}_{L^2(\hK)}.
     $$ 
     From \autoref{lem:norm_p_boundary_hK}, we have $h^{-1/2}\norm{p}_{L^2(\p\hK)}\lesssim h^{-1}\norm{p}_{L^2(\hK)}$ {which, combined with the previous inequality,  completes the proof.}
\end{proof}

\subsection{The trace  inequality for Frenet IFE functions}\label{subsec:Inverse_trace} 
In this section, we will employ the lemmas from the previous section to derive a trace inequality for IFE functions  in $\hV^m_{\beta}(\hK)$ and $\V^m_{\beta}(K)$. First, we recall that $\hV^m_{\beta}(\hK)$ consists of piecewise polynomials $\hphi$ {such that}
\begin{subequations}\label{eqn:discrete_interface_conditions}
    \begin{equation}
        \hphi\big|_{\hK^{\pm}} \in Q^m(\hK^{\pm}),\quad 
        \bb{\hphi}_{\hGamma_{K_F}} =  \bb{\hbeta\hphi_{\eta}}_{\hGamma_{K_F}} =0,
        \label{eqn:first_discrete_interface_conditions}
        \end{equation}
and 
\begin{equation}
  \left(\bb{\hbeta\frac{\p^j}{\p\eta^j} \L(\hphi)}_{\hGamma_{K_F}},v\right)_{\hGamma_{K_F}}=0,\ \forall v\in \P^m(\hGamma_{K_F}),\ 0\le j\le m-2,
  \label{eqn:second_discrete_interface_conditions}
        \end{equation}
\end{subequations}
where $\hbeta,\hK^{\pm}, \hGamma_{K_F}$ and $\L$ as summarized in \autoref{sec:space_scheme}. Here, we note the slight abuse of notation where $\hphi$ is defined on $\hK$ but the integrals in \eqref{eqn:second_discrete_interface_conditions} are defined on $\hGamma_{K_F}$ which may extend beyond $\hK$. This is not an issue since $\hphi\big|_{\hK^\pm}$ extends naturally to a  polynomial on $\mathbb{R}^{\pm}\times \mathbb{R}$.

In \cite{adjeridHighOrderGeometry2024}, we have shown that $\dim \hV^m_{\beta}(\hK)=(m+1)^2$ and described an efficient procedure to construct a basis for this IFE space. We now decompose $\hV^m_{\beta}(\hK)$ as {a direct sum of two subspaces}: a subspace of polynomials and a space of polynomials times a piecewise constant function. To describe these two subspaces, we consider two polynomial spaces as follows:
\begin{equation}
    X^{(0)} =\left\{\hphi\in Q^m(\hK)\mid \hphi_{\eta}\big|_{\hGamma_{K_F}}{=0\text{ and } } \left(\frac{\p}{\p \eta^j}\L(\hphi),v\right)_{\hGamma_{K_F}}=0,\ \forall v\in \P^m(\hGamma_{K_F}),\ 0\le j\le m-2 \right\},\label{eqn:def_X0}
\end{equation}

\begin{equation}
    X^{(1)} ={\operatorname{span}\left(\left\{\eta^j\xi^i,\ 0\le i\le m,\ 1\le j\le m \right\}\right)}.\label{eqn:def_X1}
\end{equation}
Again, we evaluate $\hphi$ {on} $\hGamma_{K_F}$ in \eqref{eqn:def_X0} {which extends outside $\hK$}. In the following lemma we show that every function in  $\hV^m_{\beta}(\hK)$ can be decomposed as a polynomial in $X^{(0)}$ plus another polynomial in $X^{(1)}$ times $\frac{1}{\hbeta}$. 

\begin{lemma}\label{lem:hV_decomposition}
    Let $\hphi \in \hV^m_{\beta}(\hK)$, then there {exist} two polynomials $\hphi_0\in X^{(0)}$ and $\hphi_1\in X^{(1)}$ such that 
    $$\hphi= \hphi^{(0)}  +\frac{1}{\hbeta}\hphi^{(1)}.$$
    Furthermore, $\hV^m_{\beta}(\hK)=X^{(0)}\oplus \frac{1}{\hbeta}X^{(1)}$ where $\oplus$ denotes direct sum. 
\end{lemma}

\begin{proof}
    The proof goes as follows: First, we show that $X^{(0)}$ and $\frac{1}{\hbeta}X^{(1)}$ are subspaces of $\hV^m_{\beta}(\hK)$. After that, {we show that their dimensions sum up to $(m+1)^2$ and prove that their intersection is trivial. }


    For a detailed proof of $\frac{1}{\hbeta}X^{(1)}\subset \hV^m_{\beta}(\hK)$, we refer the reader to Lemma 3 in \cite{adjeridHighOrderGeometry2024}. The main idea is that if $\hphi\in \frac{1}{\hbeta}X^{(1)}$ then $\hphi =\frac{1}{\hbeta}p$ for some polynomial $p$ that satisfies $p(0,\xi)=0$, { leading to} $\hphi\big|_{\hGamma_{K_F}}=0$, which, in turn, yields, $\bb{\hphi}_{\hGamma_{K_F}}{=0}$. Additionally, we have $$\bb{\hbeta\hphi_{\eta}}_{\hGamma_{K_F}} =  \bb{\hbeta\frac{\p^j}{\p\eta^j} \L(\hphi)}_{\hGamma_{K_F}}=0.$$ 
    Hence, $\hphi\in \hV^m_{\beta}(\hK)$. As for the dimension, we can use \eqref{eqn:def_X1} to deduce that $\dim(X^{(1)})=m(m+1)$.  On the other hand, to show that  $X^{(0)}\subset \hV^m_{\beta}(\hK)$, we consider a polynomial $p\in X^{(0)}$, then $\bb{p}_{\hGamma_{K_F}}{=0}$ since it is continuous. Similarly, we have 
    $$\bb{\hbeta p{_\eta}}_{\hGamma_{K_F}} = \bb{\hbeta}_{\hGamma_{K_F}} p{_\eta}\big|_{\hGamma_{K_F}} =0,
    $$ 
    $$  \left(\bb{\hbeta\frac{\p^j}{\p\eta^j} \L(\hphi)}_{\hGamma_{K_F}},v\right)_{\hGamma_{K_F}}=\bb{\hbeta}_{\hGamma_{K_F}}\left(\frac{\p^j}{\p\eta^j} \L(p),v\right)_{\hGamma_{K_F}}=0,\ \forall v\in \P^m(\hGamma_{K_F}),\ 0\le j\le m-2.$$
    Next, we show that $\dim(X^{(0)})=m+1$. For that, we consider the set of polynomials $p\in Q^m(\hK)$ that satisfy
    \begin{equation}p_{\eta}\big|_{\hGamma_{K_F}}=0,\qquad \left(\frac{\p^j}{\p\eta^j}\L(p),v\right)_{\hGamma_{K_F}}=0,\quad \forall v\in\P^m(\hGamma_{K_F}),\ 0\le j\le m-2.
        \label{eqn:X0_p_def}
    \end{equation}
    By expressing $p$ in terms of a basis of $Q^m(\hK)$, the conditions \eqref{eqn:X0_p_def} will lead to a set of $m(m+1)$ equations with $(m+1)^2$ {unknowns}.  Since these equations are linearly independent \cite{adjeridHighOrderGeometry2024}{, we have} $\dim(X^{(0)})=m+1$.

    It remains to show that $X^{(0)} \cap \frac{1}{\hbeta}X^{(1)}=\{0\}$ where $\frac{1}{\hbeta}X^{(1)}$ is the space consisting of functions $\frac{1}{\hbeta}p$ where $p\in X^{(1)}$. If $\hphi\in {X^{(0)}\cap} \frac{1}{\hbeta}X^{(1)}$, then 
    $$\hphi\big|_{\hGamma_{K_F}}=\hphi_{\eta}\big|_{\hGamma_{K_F}}=0,\qquad \left(\frac{\p^j}{\p\eta^j}\L(\hphi),v\right)_{\hGamma_{K_F}}=0,\quad \forall v\in\P^m(\hGamma_{K_F}),$$
    which leads to  $\hphi=0$ by the well-posedness of \eqref{eqn:discrete_interface_conditions}, {which is proven in \cite{adjeridHighOrderGeometry2024}. }Therefore, $\dim(X^{(0)} + \frac{1}{\hbeta}X^{(1)}) =\dim(\hV^m_{\beta}(\hK))$. 

\end{proof}


In the remainder of this article, for every function $\hphi \in \hV^m_{\beta}(\hK)$, we will use $\hphi^{(0)}$ and $\hphi^{(1)}$ to denote the {polynomials described in \autoref{lem:hV_decomposition}}. Additionally, we will assume that {$h$ is small enough for \autoref{lem:norm_p_boundary_hK} to hold}.


\begin{lemma}\label{lem:X0_decomposition}
    Let $\hphi^{(0)}\in X^{(0)}$, then $\hphi^{(0)}$ can be decomposed (uniquely) as $$
    \hphi^{(0)}(\eta,\xi)= \hphi^{(0,0)}(\xi) +\hphi^{(0,1)}(\eta,\xi) ,\qquad \forall (\eta,\xi)\in \hK,$$
    where $\hphi^{(0,0)}$ is an $m$-th degree (univariate) polynomial and $\hphi^{(0,1)}\in X^{(1)}$ with $\hphi_{\eta}^{(0,1)}(0, \xi)=0$.  
\end{lemma}

\begin{proof} {The polynomial} $\hphi^{(0)}$ can be written as 
  $$\hphi^{(0)}(\eta,\xi)=\sum_{i=0}^m\sum_{j=0}^m c_{i,j}\eta^j\xi^i=\underbrace{\sum_{i=0}^m c_{i,0}\xi^i}_{=\hphi^{(0,0)}(\xi)} +\underbrace{\sum_{i=0}^m\sum_{j=1}^m c_{i,j}\eta^j\xi^i}_{=\hphi^{(0,1)}(\eta,\xi)}.$$
   Lastly, the claim $\hphi_{\eta}^{(0,1)}(0, \xi)=0$ follows immediately from the definition \eqref{eqn:def_X0} and our construction of $\hphi^{(0,0)}$. 
\end{proof}


Next, we will show that the $L^2$ norm of $\hnabla\hphi^{(0,1)}$ on the Frenet fictitious element $\hK_F$ is bounded by the $L^2$ norm of $\hphi_{\xi}^{(0,0)}$ on $\hGamma_{K_F}$

\begin{lemma}\label{lem:phi_01_phi_00_bound}
    Let $\hphi^{(0)}\in X^{(0)}$, then 
    
\begin{equation}\norm{\hnabla \hphi^{(0,1)}}_{L^2(\hK_F)}\lessim h^{1/2} \norm{ \hphi^{(0,0)}_{\xi}}_{L^2(\hGamma_{K_F})},
    \label{eqn:phi_01_phi_00_bound}\end{equation}
    where $\hphi^{(0)}(\eta,\xi)=\hphi^{(0,0)}(\xi)+\hphi^{(0,1)}(\eta,\xi)$ is the decomposition described in \autoref{lem:X0_decomposition}. Consequently, we have  
    \begin{equation}\norm{\hnabla \hphi^{(0)}}_{L^2(\hK_F)}\lessim h^{1/2} \norm{ \hphi^{(0,0)}_{\xi}}_{L^2(\hGamma_{K_F})}.
        \label{eqn:phi_0_phi_00_bound}\end{equation}
\end{lemma}

\begin{proof}
    Recall that by definition of $X^{(0)}$, we have 
    \begin{equation}\int_{\hGamma_{K_F}}
   \frac{\p^j \L(\hphi^{(0)})}{\p\eta^j} v = 0,\qquad \forall v\in \P^{m}(\hGamma_{K_F}),\quad j=0,1,\dots,m-2.
\label{eqn:L_orthogonality}    
\end{equation}
In \cite{adjeridHighOrderGeometry2024}, we have shown that the extended jump conditions in Frenet coordinates can be expressed in a form similar to the one-dimensional IFE jump conditions discussed in \cite{adjeridUnifiedImmersedFinite2024}
    \begin{align}\frac{\p^j \L(\hphi^{(0)})}{\p\eta^j}(0,\xi)=&
        \frac{\p^{j+2} \hphi^{(0)}}{\p\eta^{j+2}}(0,\xi)\notag
        \\ &+ 
        \sum_{l=0}^{j}
        \gamma_{0,j-l}(\xi) \frac{\p^{l} \hphi_{\xi\xi}^{(0)}}{\p\eta^l}(0,\xi) +
        \gamma_{1,j-l}(\xi) \frac{\p^{l+1} \hphi^{(0)}}{\p\eta^{l+1}}(0,\xi) + 
        \gamma_{2,j-l}(\xi) \frac{\p^{l} \hphi_{\xi}^{(0)}}{\p\eta^l}(0,\xi),\label{eqn:L_expanded}
        \end{align}
        where $\norm{\gamma_{i,l}}_{C^{0}(\hGamma_{K_F})}\lesssim 1$ for all $0\le i\le 2$ and $0\le l\le j$ \cite{adjeridHighOrderGeometry2024}.  Since $\hphi^{(0,0)}$ depends only on $\xi$, we can write $\p_{\eta^j} \L(\hphi^{(0)})$ as 
\begin{equation}
    \frac{\p^j \L(\hphi^{(0,0)})}{\p\eta^j}(0,\xi)=
   \gamma_{0,j}(\xi)\hphi^{(0,0)}_{\xi\xi}(\xi)+
    \gamma_{2,j}(\xi)\hphi^{(0,0)}_{\xi}(\xi). 
    \label{eqn:L_hphi_decomp_0_1}
\end{equation}
Now, we choose $v(\xi)=\p_{\eta^{j+2}} \hphi^{(0,1)}(0,\xi)$ in \eqref{eqn:L_orthogonality} and use \eqref{eqn:L_expanded} to obtain the following equality

{\scriptsize
\begin{equation} 
    \norm{\frac{\p^{j+2}\hphi^{(0,1)}}{\p\eta^{j+2}}}^2_{L^2(\hGamma_{K_F})}
    =-\left(\frac{\p^{j+2}\hphi^{(0,1)}}{\p\eta^{j+2}},\frac{\p^j \L(\hphi^{(0,0)})}{\p\eta^j}
    +
    \sum_{l=0}^{j}\left(
\gamma_{0,j-l} \frac{\p^{l} \hphi_{\xi\xi}^{(0,1)}}{\p\eta^l} +
\gamma_{1,j-l} \frac{\p^{l+1} \hphi^{(0,1)}}{\p\eta^{l+1}} + 
\gamma_{2,j-l} \frac{\p^{l} \hphi_{\xi}^{(0,1)}}{\p\eta^l}\right)\right)_{L^2(\hGamma_{K_F})}.
\notag 
\end{equation}
}
Now, we apply Cauchy-Schwarz inequality, the 1D inverse inequality, the triangle inequality and $|\gamma_{i,l}|\lessim 1$ to obtain
\begin{align}
   \norm{\frac{\p^{j+2}\hphi^{(0,1)}}{\p\eta^{j+2}}}_{L^2(\hGamma_{K_F})}
   &\lessim \norm{\hphi_{\xi}^{(0,0)}}_{L^2(\hGamma_{K_F})}
   +\norm{\hphi_{\xi\xi}^{(0,0)}}_{L^2(\hGamma_{K_F})}
    + \norm{\frac{\p^{j+1}\hphi^{(0,1)}}{\p\eta^{j+1}}}_{L^2(\hGamma_{K_F})}
   + h^{-2}\sum_{l=0}^{j}
\norm{\frac{\p^{l} \hphi^{(0,1)}}{\p\eta^l}}_{L^2(\hGamma_{K_F})} 
 \notag 
\\ 
&\lesssim h^{-1}\norm{\hphi^{(0,0)}_{\xi}}_{L^2(\hGamma_{K_F})}
+ \norm{\frac{\p^{j+1}\hphi^{(0,1)}}{\p\eta^{j+1}}}_{L^2(\hGamma_{K_F})}
+ h^{-2}\sum_{l=0}^{j}
\norm{\frac{\p^{l} \hphi^{(0,1)}}{\p\eta^l}}_{L^2(\hGamma_{K_F})}. 
\label{eqn:phi_0_1_jp2_second_bound}
\end{align}
At this point, We use strong induction on $l$, where $0 \leq l \leq m$, to show that
\begin{equation}\norm{\frac{\p^{l}\hphi^{(0,1)}}{\p\eta^{l}}}_{L^2(\hGamma_{K_F})}\lesssim
    h^{1-l}\norm{\hphi^{(0,0)}_{\xi}}_{L^2(\hGamma_{K_F})}.
    \label{eqn:induction_conclusion}
\end{equation}
The base cases $l=0$ and $l=1$ are straightforward since $\hphi^{(0,1)}(0,\xi)=\hphi_{\eta}^{(0,1)}(0,\xi)=0$ and the induction step follows immediately from \eqref{eqn:phi_0_1_jp2_second_bound}. Since $\hphi^{(0,1)}$ is an $m$-th degree polynomial in $\eta$, we can write it as
$$\hphi^{(0,1)}(\eta,\xi) = \sum_{l=0}^{m} \frac{\eta^l}{l!} \frac{\p^l \hphi^{(0,1)}}{\p\eta^l}(0,\xi).$$
The $L^2$ norm over $\hK_F=[-h,h]\times \hGamma_{\hK_F}$ with \eqref{eqn:induction_conclusion} yields  
\begin{align}
    \norm{\hphi^{(0,1)}}_{L^2(\hK_F)} \lessim \sum_{l=0}^{m} h^{l+1/2} \norm{\frac{\p^l \hphi^{(0,1)}}{\p\eta^l}}_{L^2(\hGamma_{\hK_F})}  \lesssim 
    h^{3/2} \norm{\hphi^{(0,0)}_{\xi}}_{L^2(\hGamma_{K_F})} .
    \notag
     \end{align}
     We apply the classical inverse inequality on $Q^m(\hK_F)$ to the previous estimate
     \begin{equation}\norm{\hnabla \hphi^{(0,1)}}_{L^2(\hK_F)}\lessim 
     h^{-1}\norm{\hphi^{(0,1)}}_{L^2(\hK_F)}
     \lesssim h^{1/2}\norm{\hphi^{(0,0)}_{\xi}}_{L^2(\hGamma_{K_F})},
     \label{eqn:phi_1_bounded_by_phi_0}
    \end{equation}
    which proves \eqref{eqn:phi_01_phi_00_bound}. Lastly, to derive \eqref{eqn:phi_0_phi_00_bound}, we use the triangle inequality  and \eqref{eqn:phi_1_bounded_by_phi_0} to obtain 
     \begin{align}
        \norm{\hnabla \hphi^{(0)}}_{L^2(\hK_F)}&\le 
        \norm{\hnabla \hphi^{(0,0)}}_{L^2(\hK_F)}+
        \norm{\hnabla \hphi^{(0,1)}}_{L^2(\hK_F)}\\ 
        & \simeq h^{1/2}\norm{\hnabla \hphi^{(0,0)}}_{L^2(\hGamma_{K_F})}+
        \norm{\hnabla \hphi^{(0,1)}}_{L^2(\hK_F)}
        \lesssim h^{1/2}\norm{\hnabla \hphi^{(0,0)}}_{L^2(\hGamma_{K_F})},
        \notag
     \end{align}
where the last line follows from the fact that $\hphi^{(0,0)}$ is a function of $\xi$ solely, and from \eqref{eqn:phi_1_bounded_by_phi_0}.

\end{proof}

The previous lemmas show that the polynomials in $X^{(0)}$ are characterized and controlled by their traces on $\hGamma_{K_F}$. On the other hand, the polynomials in $X^{(1)}$ vanish on $\hGamma_{K_F}$. This observation allows us to bound the norm of a polynomial in $X^{(0)}\oplus X^{(1)}$ by the norm of its projections onto $X^{(0)}$ and $X^{(1)}$. 

\begin{lemma}\label{lem:norm_equiv_X0_X1_hK_s} 
    Let $s\in \{+,-\}$, then 
     \begin{equation}
    \norm{\hnabla \hphi^{(0)}}_{L^2(\hK_F^s)}+\norm{\hnabla \hphi^{(1)}}_{L^2(\hK_F^s)}\lesssim 
   \norm{ \hnabla \hphi^{(0)}+\hnabla \hphi^{(1)}}_{L^2(\hK_F^s)},\qquad \forall \hphi^{(0)}\in X^{(0)},\ \forall \hphi^{(1)} \in X^{(1)}.
   \label{eqn:norm_equiv_X0_X1_V_hK_s}
     \end{equation}
 \end{lemma}

 \begin{proof}
    Recall that $\hK_F^+=[0,h]\times [\xi_0,\xi_1]$ and  $\hK_F^-=[-h,0]\times [\xi_0,\xi_1]$, then by the trace inequality for polynomials and \autoref{lem:xi1_minus_xi0}, we have 

    \begin{equation}
        \norm{ \hnabla \hphi^{(0)}+\hnabla \hphi^{(1)}}_{L^2(\hK_F^s)} \ge 
        \norm{  \hphi^{(0)}_{\xi}+ \hphi^{(1)}_{\xi}}_{L^2(\hK_F^s)}\gtrsim   h^{1/2}\norm{  \hphi^{(0)}_{\xi}+ \hphi^{(1)}_{\xi}}_{L^2(\hGamma_{K_F})}\label{eqn:phi_xi_boundary}
    \end{equation}
    From \eqref{eqn:def_X1}, we have $\hphi^{(1)}_{\xi}=0$ on $\hGamma_{K_F}$, and from \autoref{lem:X0_decomposition}, we have $\hphi^{(0)}_{\xi}=\hphi^{(0,0)}_{\xi}$ on $\hGamma_{K_F}$. Hence, \eqref{eqn:phi_xi_boundary} combined with \autoref{lem:phi_01_phi_00_bound} yields 
    \begin{equation}
        \norm{ \hnabla \hphi^{(0)}+\hnabla \hphi^{(1)}}_{L^2(\hK_F^s)} \gtrsim  h^{1/2}\norm{  \hphi^{(0,0)}_{\xi}}_{L^2(\hGamma_{K_F})}
        \gtrsim   \norm{\nabla  \hphi^{(0)}}_{L^2(\hK_F)}.
        \label{eqn:phi_0_bounded} 
    \end{equation}
    Next, we use the triangle inequality and \eqref{eqn:phi_0_bounded} to obtain 
    \begin{align}
        \norm{\hnabla \hphi^{(0)}}_{L^2(\hK_F^s)}+\norm{\hnabla \hphi^{(1)}}_{L^2(\hK_F^s)}&\le  2
        \norm{ \hnabla \hphi^{(0)}}_{L^2(\hK_F^s)} + 
       \norm{ \hnabla \hphi^{(0)}+\hnabla \hphi^{(1)}}_{L^2(\hK_F^s)}  \notag \\ 
      &\lessim 
      \norm{ \hnabla \hphi^{(0)}+\hnabla \hphi^{(1)}}_{L^2(\hK_F^s)}.
         \end{align}

 \end{proof}

    

The previous lemma is extremely important since it will allow us to bound norms of IFE functions in terms of norms of polynomials, which are easier to handle. However, the lemma deals with norms over the Frenet fictitious sub-elements $\hK_F^{\pm}$ instead of the Frenet interface element $\hK$ or its sub-elements $\hK^{\pm}$. Ideally, one would try to prove \autoref{lem:norm_equiv_X0_X1_hK_s} with norms over $\hK^s$ directly instead of $\hK_F^s$. However, this approach might be greatly complicated. Instead, we adopt an indirect approach that relies on the following two lemmas.

First, we show that on a Frenet interface element $\hK$, there exists a triangle completely contained in the closure of $\hK^+$ or $\hK^-$ with side-lengths comparable to $h$, which leads to the following lemma.

\begin{lemma}\label{lem:stable_extensions_poly}
    Let $K$ be an interface element, then there exists $s^*\in \{-,+\}$ such that 

    \begin{equation}
        \norm{p}_{L^2(\hK_F^{s^*})} \lesssim
        \norm{p}_{L^2(\hK^{s^*})}, \qquad \forall p\in Q^m(\hK_F^{s^*}),
        \label{eqn:type_plus_minus_equivalence_poly}
    \end{equation}  
    whenever $\operatorname{diam}(K)<h_4$ where $h_4>0$ is independent of the mesh size and the relative position of the interface to $K$. 
\end{lemma}

\begin{proof} Let $A_1, A_2, A_3$ and $A_4$ be the vertices of $K$ and let $\hat{A}_i=R(A_i)$ for $1\le i\le 4$. We need to consider two cases: 
\begin{itemize}
    \item Case 1: Three of the vertices $\hat{A}_i, ~1 \leq i \leq 4$ are in the closure of $\hK^+$ or $\hK^{-}$. To be specific and without loss of generality, we assume that $\hA_1,\hA_2,\hA_3\in \hK^{s^*}$ for an $s^*\in \{+,-\}$ as illustrated in \autoref{subfig:case_1_fig}. 
    
    Now, let $\check{A}_i=\cR(A_i)=T(\hat{A_i})$ for $i=1,2,3$, then by \autoref{lem:Tx_minus_x}, we have $\norm{\hat{A}_i-\check{A}_i}\lesssim h^2$.  Furthermore, the curved triangle $\hat{M}=\cdelta \hA_1\hA_2\hA_3$ is contained in $\check{M}+B(\mathbf{0},C_1h^2)$ where $\check{M}=\triangle \check{A}_1\check{A}_2\check{A}_3$ for some $C_1>0$ independent of mesh size and the relative position of the interface, and $\check{M}$ is also contained in $\hat{M}+B(\mathbf{0},C_1h^2)$.     On the other hand, it follows from \eqref{eqn:jacobian_cP_cR_good} that the inradius $\rho$ of $\check{M}$ satisfies $\rho\ge C_2 h$ for some $C_2$ independent of the mesh size and the relative position of the interface. Hence,    the homothetic image $\check{M}_0$ of $\check{M}$ with scaling $\frac{1}{2}$ centered at $\check{G}_{\check{M}}$, the incenter of $\check{M}$, is contained in $\hat{M}$ for $h<\frac{C_2}{4C_1}$. To see this, consider a point $\check{\x}\in \check{M}_0$, then $\norm{\check{\x}-\check{\y}}\ge \rho/2$ for any $\check{\y}\in \p \check{M}$. Therefore, given any $\hat{\z}\in \p \check{M}$ and $\check{\x}\in \check{M}_0$, we have 
    \[ \norm{\hat{\z}-\check{\x}}\ge \frac{\rho}{2} - C_1 h^2 \ge  \frac{C_2}{2}h - \frac{C_2}{4}h= \frac{C_2}{4}h>0.\notag
    \]
    Hence, $\check{M}_0\subset \hat{M}$ for $h$ small enough.

    Next, let $p\in Q^m(\hK^{s^*}_F)$, then by our observation above and the homothetic scaling argument \cite{xiaoHighorderExtendedFinite2020}, we have 
    \begin{equation}
        \norm{p}_{L^2(\hK^{s^*})}\ge 
        \norm{p}_{L^2(\hat{M})} \ge  \norm{p}_{L^2(\check{M}_0)}
        \gtrsim   \norm{p}_{L^2(\check{M})}.
        \label{eqn:Ks_star_to_check_M}
    \end{equation}
    On the other hand, consider the homothetic image $\check{M}_1$ of $\check{M}$ with scaling $2\diam(\hK_F^{s^*})/\rho \simeq 1$ and center $\check{G}_{\check{M}}$, then $\hat{K}_F^{s^*}\subset \check{M}_1$.
 Therefore, by the homothetic scaling argument and \eqref{eqn:Ks_star_to_check_M}, we  have 
    \begin{equation}
        \norm{p}_{L^2(\hK^{s^*})}\gtrsim  \norm{p}_{L^2(\check{M}_1)} \gtrsim \norm{p}_{L^2(\hat{K}_F^{s^*})}.
    \end{equation}

    \item Case 2: There are two vertices in $\hK^-$ and two vertices in $\hK^+$. Without loss of generality, we assume that $\hA_1\in \hK^{-}$ and $\hA_2\in \hK^+$ as shown in \autoref{subfig:case_2_fig}. Let $D$ be the intersection point of the line segment $\hGamma_{K_F}$ and the curved edge connecting $\hA_1$ and $\hA_2$. Again, without loss of generality, we assume that $|D\hA_2|\ge \frac{1}{2}|\hA_1\hA_2|$, then $|D\hA_2|\simeq h$. Therefore, by following the same steps in the previous case, we obtain 
    \begin{equation}
        \norm{p}_{L^2(\hK^+)} \ge 
        \norm{p}_{L^2(\cdelta \hA_2\hA_3D)}\gtrsim 
        \norm{p}_{L^2(\triangle \check{A}_2\check{A}_3 \check{D})}
        \gtrsim
        \norm{p}_{L^2(\hK_F^+)}       
    \end{equation}       
\end{itemize}

\begin{figure}[htbp]
    \center
    \begin{subfigure}{.5\textwidth}
    \center
    \includegraphics[scale=.7]{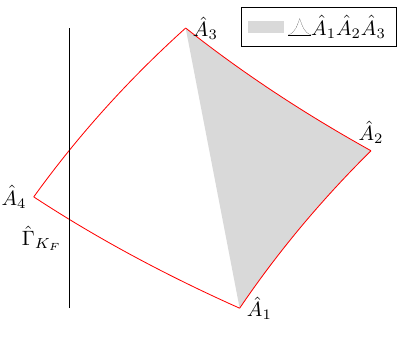}
    \caption{Case 1}

    \label{subfig:case_1_fig}

\end{subfigure}\begin{subfigure}{.5\textwidth}
    \center
    \includegraphics[scale=.7]{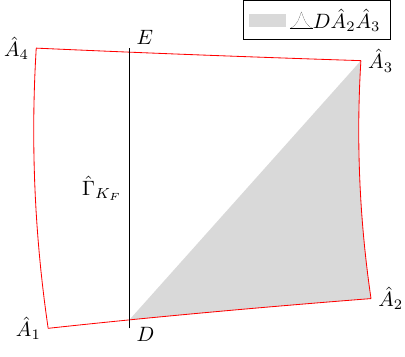}
    \caption{Case 2}
    \label{subfig:case_2_fig}
\end{subfigure}
\caption{An illustration of the curved triangles described in the proof of \autoref{lem:stable_extensions_poly}}
\end{figure}

\end{proof}

{From now on in this section, we assume that the mesh $\T_h$ is sufficiently fine such that the results 
from the previous lemmas are all valid.} Next, we use \autoref{lem:stable_extensions_poly} to the Frenet IFE functions to obtain the following lemma. 
\begin{lemma}\label{lem:stable_extensions}
    Let $K$ be an interface element, then we have 
    \begin{equation}
        \norm{\hbeta \hnabla \hphi}_{L^2(\hK_F)} \lesssim
        \norm{\hbeta\hnabla \hphi}_{L^2(\hK)}, \qquad \forall \hphi\in \hV^m_{\beta}(\hK_F).
        \label{eqn:type_minus_equivalence}
    \end{equation}
    or 
    \begin{equation}
        \norm{\hnabla \hphi}_{L^2(\hK_F)} \lesssim
        \norm{\hnabla \hphi}_{L^2(\hK)}, \qquad \forall \hphi\in \hV^m_{\beta}(\hK_F).
        \label{eqn:type_plus_equivalence}
    \end{equation}
\end{lemma}

\begin{proof}
    First, let us assume that \eqref{eqn:type_plus_minus_equivalence_poly} holds for $s^*=+$. Then, by \autoref{lem:hV_decomposition}, \autoref{lem:norm_equiv_X0_X1_hK_s}, \autoref{lem:stable_extensions_poly} with $s^*=+$, and the assumption that $\beta^- \leq \beta^+$,   we have
    \begin{align*}
    \norm{\hbeta \hnabla \hphi}_{L^2(\hK_F)}
  &= \norm{\hbeta \hnabla \hphi^{(0)} + \hnabla \hphi^{(1)}}_{L^2(\hK_F)}\notag \le \beta^+\norm{\hnabla \hphi^{(0)}}_{L^2(\hK_F)} + \norm{\hnabla \hphi^{(1)}}_{L^2(\hK_F)} \\
  & \lesssim  \beta^+\norm{\hnabla \hphi^{(0)}}_{L^2(\hK_F^+)} + \norm{\hnabla \hphi^{(1)}}_{L^2(\hK_F^+)} \lesssim \norm{\beta^+ \hnabla \hphi}_{L^2(\hK^+_F)} \\ 
  &\lesssim \norm{\beta^+ \hnabla \hphi}_{L^2(\hK^+)} \lesssim \norm{\hbeta \hnabla \hphi}_{L^2(\hK^+)} \lessim \norm{\hbeta \hnabla \hphi}_{L^2(\hK)}. \notag 
\end{align*}

On the other hand, if \eqref{eqn:type_plus_minus_equivalence_poly} holds for $s^*=-$, then, 
similarly, by \autoref{lem:hV_decomposition}, \autoref{lem:norm_equiv_X0_X1_hK_s}, \autoref{lem:stable_extensions_poly} with $s^*=-$, and the assumption that $\beta^- \leq \beta^+$, we have 
\begin{align*}
    \norm{ \hnabla \hphi}_{L^2(\hK_F)}
  &= \norm{ \hnabla \hphi^{(0)} +\frac{1}{\hbeta} \hnabla \hphi^{(1)}}_{L^2(\hK_F)} \le \norm{\hnabla \hphi^{(0)}}_{L^2(\hK_F)} +\frac{1}{\beta^-} \norm{\hnabla \hphi^{(1)}}_{L^2(\hK_F)} \\
  & \lesssim  \norm{\hnabla \hphi^{(0)}}_{L^2(\hK_F^-)} +\frac{1}{\beta^-} \norm{\hnabla \hphi^{(1)}}_{L^2(\hK_F^-)}  \lesssim  \norm{\hnabla \hphi}_{L^2(\hK_F^-)}  \lesssim\norm{\hnabla \hphi}_{L^2(\hK^-)} \lesssim \norm{\hnabla \hphi}_{L^2(\hK)}.\notag 
\end{align*}

\commentout{
        \begin{align}
            \norm{ \hnabla \hphi}_{L^2(\hK_F)}
          &= \norm{ \hnabla \hphi^{(0)} +\frac{1}{\hbeta} \hnabla \hphi^{(1)}}_{L^2(\hK_F)}\notag \\ 
          & \le \norm{\hnabla \hphi^{(0)}}_{L^2(\hK_F)} +\frac{1}{\beta^-} \norm{\hnabla \hphi^{(1)}}_{L^2(\hK_F)}\tag{Triangle inequality and $\beta^-\le \beta^+$} \\
          & \lesssim  \norm{\hnabla \hphi^{(0)}}_{L^2(\hK_F^-)} +\frac{1}{\beta^-} \norm{\hnabla \hphi^{(1)}}_{L^2(\hK_F^-)} \tag{$\norm{p}_{\hK_F^+}\simeq\norm{p}_{\hK_F}$ for polynomials $p$}  \\ 
          & \lesssim  \norm{\hnabla \hphi}_{L^2(\hK_F^-)}  \tag{Using \autoref{lem:norm_equiv_X0_X1_hK_s}} \\ 
          &\lesssim\norm{\hnabla \hphi}_{L^2(\hK^-)} 
          \tag{Using \autoref{lem:stable_extensions_poly} with $s^*=-$}
          \notag\\ 
            &\lesssim \norm{\hnabla \hphi}_{L^2(\hK)}.\notag 
        \end{align}
}

\end{proof}


At this point, we are ready to extend the results of \autoref{lem:norm_equiv_X0_X1_hK_s} to IFE functions on the Frenet element $\hK$ in the following lemma. 
\begin{lemma}\label{lem:norm_equiv_X0_X1_V} 
    For every $\hphi\in \hV^{m}_{\beta}(\hK)$ with $\hphi=\hphi^{(0)}+\hbeta^{-1}\hphi^{(1)}$ as stated in \autoref{lem:hV_decomposition}, the following estimate holds:
    \begin{equation}
    \beta^+\norm{\hnabla \hphi^{(0)}}_{L^2(\hK)}+\norm{\hnabla \hphi^{(1)}}_{L^2(\hK)}\lesssim 
   \frac{\beta^+}{\sqrt{\beta^-}}\norm{\sqrt{\hbeta} \hnabla \hphi}_{L^2(\hK)}.
   \label{eqn:norm_equiv_X0_X1_V}
    \end{equation}
     
\end{lemma}

\begin{proof}
    The proof is divided into two parts. First, we prove the lemma in the case where \eqref{eqn:type_minus_equivalence} holds, then in the case where \eqref{eqn:type_plus_equivalence} holds. 
    Firstly, If \eqref{eqn:type_minus_equivalence} holds, then, by \autoref{lem:norm_equiv_X0_X1_hK_s},  we have 
    \begin{align}\label{eqn:case_pos_x0_x1} 
    \begin{split}
        \beta^+\norm{\hnabla \hphi^{(0)}}_{L^2(\hK)}+\norm{\hnabla \hphi^{(1)}}_{L^2(\hK)}
        &\le  \beta^+\norm{\hnabla \hphi^{(0)}}_{L^2(\hK_F)}+\norm{\hnabla \hphi^{(1)}}_{L^2(\hK_F)} \\ 
        &\lesssim \beta^+\norm{\hnabla \hphi^{(0)}}_{L^2(\hK_F^+)}+\norm{\hnabla \hphi^{(1)}}_{L^2(\hK_F^+)} \\ 
        &\lesssim  \norm{\hbeta \hnabla \hphi}_{L^2(\hK_F^+)} \lesssim \norm{\hbeta \hnabla \hphi}_{L^2(\hK_F)} \lesssim \norm{\hbeta \hnabla \hphi}_{L^2(\hK)} \\
        & \lesssim \sqrt{\hbeta^+} \norm{\sqrt{\hbeta} \hnabla \hphi}_{L^2(\hK)}  \lesssim \frac{\beta^+}{\sqrt{\beta^-}}\norm{\sqrt{\hbeta} \hnabla \hphi}_{L^2(\hK)}
    \end{split}
    \end{align}
which proves \eqref{eqn:norm_equiv_X0_X1_V} in this case.     
Secondly, assume \eqref{eqn:type_plus_equivalence} holds. We start from
    \begin{align}\label{eqn:potentially_complicated_x0_x1} 
    \begin{split}
        \beta^+\norm{\hnabla \hphi^{(0)}}_{L^2(\hK)}+\norm{\hnabla \hphi^{(1)}}_{L^2(\hK)}
        &\le  \beta^+\norm{\hnabla \hphi^{(0)}}_{L^2(\hK_F)}+\norm{\hnabla \hphi^{(1)}}_{L^2(\hK_F)} \\ 
        &\lesssim \beta^+\norm{\hnabla \hphi^{(0)}}_{L^2(\hK_F^+)}+\norm{\hnabla \hphi^{(1)}}_{L^2(\hK_F^+)} \\ 
        &\lesssim \beta^+\left(\norm{\hnabla \hphi^{(0)}}_{L^2(\hK_F^+)}+\norm{\frac{1}{\hbeta}\hnabla \hphi^{(1)}}_{L^2(\hK_F^+)}\right)\\ 
       & \lesssim  \beta^+\left(\norm{\hnabla \hphi^{(0)}}_{L^2(\hK_F)}+\norm{\frac{1}{\hbeta}\hnabla \hphi^{(1)}}_{L^2(\hK_F)}\right)        \\ 
        &\lesssim \beta^+ \norm{\hnabla \hphi}_{L^2(\hK_F)},
    \end{split}    
    \end{align}
where the last line follows from applying \autoref{lem:norm_equiv_X0_X1_hK_s} to both $\hK^{+}_F$ and $\hK^{-}_F$ as follows

\begin{align}
\left(\norm{\hnabla \hphi^{(0)}}_{L^2(\hK_F)}+\norm{\frac{1}{\hbeta}\hnabla \hphi^{(1)}}_{L^2(\hK_F)}\right)^2
    &\lesssim \norm{\hnabla \hphi^{(0)}}_{L^2(\hK_F)}^2+\norm{\frac{1}{\hbeta}\hnabla \hphi^{(1)}}_{L^2(\hK_F)}^2\notag \\ 
   & \lesssim \sum_{s\in\{+,-\}} {\left(\norm{\hnabla \hphi^{(0)}}_{L^2(\hK_F^s)}^2+\norm{\frac{1}{\beta^s}\hnabla \hphi^{(1)}}_{L^2(\hK_F^s)}^2\right)}\notag \\ 
& \lesssim 
\sum_{s\in\{+,-\}} 
\norm{\hnabla\hphi}_{L^2(\hK_F^s)}^2 \lesssim \norm{\hnabla\hphi}_{L^2(\hK_F)}^2.\notag 
\end{align}
Thus, applying \eqref{eqn:type_plus_equivalence} to \eqref{eqn:potentially_complicated_x0_x1}, we have
    \begin{equation}
        \beta^+\norm{\hnabla \hphi^{(0)}}_{L^2(\hK)}+\norm{\hnabla \hphi^{(1)}}_{L^2(\hK)}  \lesssim \beta^+ \norm{\hnabla \hphi}_{L^2(\hK_F)}  \lesssim \beta^+\norm{ \hnabla \hphi}_{L^2(\hK)} \le \frac{\beta^+}{\sqrt{\beta^-}} \norm{\sqrt{\hbeta} \hnabla \hphi}_{L^2(\hK)}\label{eqn:case_neg_x0_x1}
    \end{equation}
which also proves \eqref{eqn:norm_equiv_X0_X1_V}. 

\end{proof}

Now, we are ready to establish a trace theorem for Frenet IFE functions. 

\begin{theorem}\label{thm:trace_inverse_hK}
    Let $K$ be an interface element and let $\hV^{m}_{\beta}(\hK)$ be the piecewise polynomial Frenet IFE space defined in \eqref{eqn:discrete_interface_conditions}, then we have {the following trace inequality on $\hK$}
    \begin{equation}
        \norm{\hbeta \hnabla \hphi}_{L^2(\p\hK)} \lesssim 
        \frac{\beta^+}{\sqrt{\beta^-}} h^{-1/2}\norm{\sqrt{\hbeta} \hnabla \hphi}_{L^2(\hK)},\qquad \forall
        \hphi\in \hV^{m}_{\beta}(\hK) {.}
        \label{eqn:trace_inverse_hK}
    \end{equation}
\end{theorem}

\begin{proof}
    Using \autoref{lem:hV_decomposition}, we can write $\hphi$ as $\hphi^{(0)}+\frac{1}{\hbeta}\hphi^{(1)}$ where $\hphi^{(0)}$ and $\hphi^{(1)}$ are polynomials. First,

    \begin{align*}  \norm{\hbeta \hnabla \hphi }_{L^2(\p\hK)}&\le 
    \norm{\hbeta\hnabla\hphi^{(0)}}_{L^2(\p\hK)}+ \norm{\hnabla\hphi^{(1)}}_{L^2(\p\hK)}
    \\ &\le 
    \beta^{+}\norm{\hnabla\hphi^{(0)}}_{L^2(\p\hK)}+ \norm{\hnabla\hphi^{(1)}}_{L^2(\p\hK)}.
    \end{align*}
Then, we use the trace inequality from \eqref{eqn:trace_inverse_Q_hK} and \autoref{lem:norm_equiv_X0_X1_V} to obtain 
\begin{align*}
\norm{\hbeta \hnabla \hphi }_{L^2(\p\hK)}
\lesssim h^{-1/2}\left( \beta^{+}\norm{\hnabla\hphi^{(0)}}_{L^2(\hK)}+ \norm{\hnabla\hphi^{(1)}}_{L^2(\hK)}\right)
\lesssim h^{-1/2}\frac{\beta^+}{\sqrt{\beta^-}}\norm{\sqrt{\hbeta}\hnabla\hphi}_{L^2(\hK)}
\end{align*}
which proves \eqref{eqn:trace_inverse_hK}. 
\end{proof}



We conclude  this section by recalling that if $\phi\in \V^{m}_{\beta}(K)$, then there exists $\hphi\in \hV^{m}_{\beta}(\hK)$ such that $\phi=\hphi\circ R$. Therefore, we obtain a trace inequality for
functions in $\V^{m}_{\beta}(K)$.

\begin{theorem}\label{thm:trace_inverse_K}
    Let $K$ be an interface element, then 
    \begin{equation}
        \norm{\beta \nabla \phi }_{L^2(\p K)} \lesssim 
        \frac{\beta^+}{\sqrt{\beta^-}} h^{-1/2}\norm{\sqrt{\beta} \nabla \phi}_{L^2(K)},\qquad \forall
        \phi\in \V^{m}_{\beta}(K).
        \label{eqn:trace_inverse_K}
    \end{equation}
\end{theorem}

\begin{proof}
    First, we recall from \cite{adjeridHighOrderGeometry2024} that $\norm{\nabla\phi}\simeq \norm{\hnabla\hphi\circ R}$ and that $\norm{DR}\simeq 1$, {where $\norm{\cdot}$ is the Euclidean norm on $\mathbb{R}^2$.} Consequently, we have 
$$\norm{\beta \nabla \phi }_{L^2(\p K)}
\lesssim  \norm{\beta \hnabla \hphi \circ R }_{L^2(\p K)} \lesssim \norm{\hbeta \hnabla \hphi }_{L^2(\p \hK)} 
$$ 
where the last inequality follows from a change of variables. Now, applying \autoref{thm:trace_inverse_hK} to bound the right most term and using the change of variable $\hx= R(\x)$, we have
\begin{align*}
\norm{\beta \nabla \phi}_{L^2(\p K)}&\lesssim  \frac{\beta^+}{\sqrt{\beta^-}} h^{-1/2}\norm{\sqrt{\hbeta} \hnabla \hphi}_{L^2(\hK)} \lesssim 
\frac{\beta^+}{\sqrt{\beta^-}} h^{-1/2}\norm{\sqrt{\beta} \hnabla \hphi\circ R}_{L^2(K)}  \\ 
&\lesssim \frac{\beta^+}{\sqrt{\beta^-}} h^{-1/2}\norm{\sqrt{\beta} \nabla \phi}_{L^2(K)}.
\end{align*}
which proves the trace inequality stated in \eqref{eqn:trace_inverse_K}. 
\end{proof}

In summary, combining the result of \autoref{thm:trace_inverse_K} with the standard trace inequality for polynomials on non-interface elements, we know
that the following estimate holds:

\begin{equation} \norm{\beta \nabla \phi }_{L^2(\p K)} \le C_t 
\frac{\beta^+}{\sqrt{\beta^-}} h^{-1/2}\norm{\sqrt{\beta} \nabla \phi}_{L^2(K)},\ \forall \phi\in \V^{m}_{\beta}(\T_h),\ \forall K\in \T_h,
\label{eqn:global_trace_inequality}
\end{equation}
{where } $C_t>0$ is independent of the mesh size, the relative position of the interface to the mesh, and the coefficients $\beta^{\pm}$.

\section{The analysis of the Frenet IFE method} \label{sec:analysis}

In this section, we provide an \emph{a priori} error estimate for the Frenet IFE method \eqref{eqn:compact_discrete_weak_form} which was originally introduced in \cite{adjeridHighOrderGeometry2024}.
Our analysis follows the steps for obtaining the error estimates for the classical SIPDG method \cite{riviereDiscontinuousGalerkinMethods2008}. A similar approach can be found in the analysis of the partially penalized IFE method \cite{guoHigherDegreeImmersed2019,linPartiallyPenalizedImmersed2015a}. However, particularities in our methods warrant a few details in the derivation of the error estimates. We start from the assumption that the mesh $\T_h$ is sufficiently fine such that the results from the previous lemmas are all valid. Then, we consider the following norms on $\V^m_{\beta}(\T_h)$
\begin{equation}
\norm{v}_h^2= \sum_{K\in \T_h}\norm{\sqrt{\beta}\nabla v}_{L^2(K)}^2+ \frac{\gamma \sigma_0}{h}\sum_{e\in \E_h} \norm{\bb{v}_e}_{L^2(e)}^2, \label{eqn:def_norm_h}
\end{equation}
\begin{equation}
    \vertiii{v}_h^2= \norm{v}_h^2+ \frac{h}{\sigma_0 \gamma}\sum_{e\in \E_h} \norm{\cc{\beta \nabla v\cdot \n}_e}_{L^2(e)}^2, \label{eqn:def_energy_norm}
\end{equation}
and we use $P_hu$ to denote the $L^2$ projection of a function $u\in L^2(\O)$ onto $\V^m_{\beta}(\T_h)$. That is, 
$$P_hu\in \V^m_{\beta}(\T_h),\qquad \left(u-P_h u,v_h\right)_{\O} =0,\ \forall v_h\in \V^m_{\beta}(\T_h).$$
From the classical finite element estimates {on a non-interface element $K$ }\cite{ciarletFiniteElementMethod2002}, we know that 
\begin{equation}\left|u-P_hu\right|_{H^{i}(K)}\lesssim h^{m+1-i}|u|_{H^{m+1}(K)},\qquad 0\le i\le m.
    \label{eqn:proj_error_non_int}
\end{equation}
    On the other hand, if $K$ is an interface element, we have shown in \cite{adjeridHighOrderGeometry2024} that   
\begin{equation}\left|u-P_hu\right|_{\H^{i}(K)}\lesssim h^{m+1-i}\norm{u}_{\H^{m+1}(K_F)},\qquad 0\le i\le m,
\label{eqn:proj_error_on_KF}
\end{equation}
where $K_F$ is the fictitious element in the tubular neighborhood of $\Gamma$ as shown in \autoref{fig:K_F_illustration}. The next natural step would be to sum \eqref{eqn:proj_error_non_int} and \eqref{eqn:proj_error_on_KF} over all elements $K\in \T_h$, but $K_F$ is larger than $K$ in general. For this reason, we need to show that $K_F$ only intersect a (uniformly) finite number of elements.

\begin{figure}[htbp]
\centering 
\includegraphics[scale=.8]{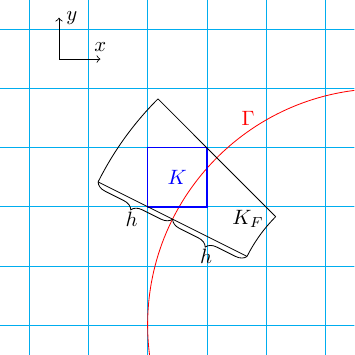}
\caption{{An illustration interface element $K$ (blue) and its fictitious element $K_F$ (black)}}
\label{fig:K_F_illustration}
\end{figure}

\begin{lemma}
    {Let $K$ be an interface element and let $K_F$ be its fictitious element, then $K_F$ intersects at most $7^2$ elements in $\T_h$.}
    \label{lem:finite_overlap}
\end{lemma}

\begin{proof}
    Let $\x\in K_F$, then $\x=\g(\xi)+\eta \n(\xi)$ for some $\eta\in [-h,h]$ and $\xi\in[\xi_0,\xi_1]$. By construction of $K_F$, {there exists
    $\eta'\in[-h,h]$ such that $\y=\g(\xi)+\eta' \n(\xi)\in K$}.
     Then, $\norm{\x-\y}=|\eta-\eta'|\le 2h$. In conclusion, we have $K_F\subset K +\overline{B}(\mathbf{0},2h)$. 
    Next, let $\omega_K$ be the union of the $7^2$ elements surrounding $K$ (including $K$), then $K +\overline{B}(\mathbf{0},2h)\subset \omega_K$. Consequently, $K_F$ intersects at most $7^2$ elements.
\end{proof}

Now, we are ready to state the error estimate on the whole mesh $\T_h$. 

\begin{lemma}\label{lem:global_proj_error}
    Let $u\in \H^{m+1}(\O,\Gamma;\beta)$ and let $P_hu\in \V^m_{\beta}(\T_h)$ be its $L^2$ projection, then 
    \begin{equation}
        \left|u-P_h u\right|_{\H^{i}(\O)}\lesssim h^{m+1-i}\norm{u}_{\H^{m+1}(\O)},\qquad 0\le i\le m.
        \label{eqn:global_proj_error}
    \end{equation}
\end{lemma}

\begin{proof}
This follows immediately from \eqref{eqn:proj_error_non_int}, \eqref{eqn:proj_error_on_KF} and \autoref{lem:finite_overlap}.
\end{proof}

\begin{lemma}\label{lem:proj_energy_error}
Let $u\in \H^{m+1}(\O,\Gamma;\beta)$ and let $P_hu\in \V^m_{\beta}(\T_h)$ be its $L^2$ projection, then 

\begin{equation}
    \vertiii{u-P_h u}_h \lesssim \frac{\beta^+}{\sqrt{\beta^-}} h^m \norm{u}_{\H^{m+1}(\O)}.
    \label{eqn:proj_energy_error}
\end{equation}
\end{lemma}
\begin{proof} For conciseness, we let $w=u-P_hu$. Then, by \autoref{lem:global_proj_error}, we have 

\begin{equation}
    \sum_{K\in \T_h} \norm{\sqrt{\beta} \nabla w}_{L^2(K)} \le \sqrt{\beta^{+}}\sum_{K\in \T_h} \norm{ \nabla w}_{L^2(K)} \lesssim \sqrt{\beta^+} h^m  \norm{u}_{\H^{m+1}(\O)}.\label{eqn:nabla_w_bound}
\end{equation}
{Since $w|_K\in H^1(K)$ for any element $K$, the Sobolev trace theorem and \autoref{lem:global_proj_error} yields }
\begin{align}
\sum_{e\in \E_h} \norm{\bb{w}_e}_{L^2(e)}
&\lesssim \sum_{K\in \T_h} \norm{w\big|_K}_{L^2(\p K)} \lesssim \sum_{K\in \T_h} \left(h^{-1/2}\norm{w}_{L^2( K)} +h^{1/2}\norm{\nabla w}_{L^2( K)} \right) \notag\\
&\lesssim h^{m+1/2} \norm{u}_{\H^{m+1}(\O)}.
\label{eqn:jump_w_bound}
\end{align}
Similarly, we have 
\begin{equation}
    \sum_{e\in \E_h} \norm{ \cc{\beta\nabla w\cdot \n}_e}_{L^2(e)} \lesssim \beta^+h^{m-1/2}\norm{u}_{\H^{m+1}(\O)}.
    \label{eqn:jump_nabla_w_bound}
\end{equation}
Finally, we combine \eqref{eqn:nabla_w_bound}-\eqref{eqn:jump_nabla_w_bound} and \eqref{eqn:def_energy_norm} to obtain
$$\vertiii{w}_h^2\lesssim h^{2m}\norm{u}_{\H^{m+1}(\O)}^2 \left( \beta^+ +\sigma_0 \gamma +\frac{1}{\sigma_0 \gamma}(\beta^+)^2\right)\lesssim \frac{(\beta^+)^2}{\beta^-} h^{2m}\norm{u}_{\H^{m+1}(\O)}^2,
$$
where the last inequality follows from $\gamma=\frac{(\beta^+)^2}{\beta^-}$. The proof is completed by taking the square root of the quantities on of the above estimate. 
\end{proof}

{The following is about the coercivity of the bilinear form used in the Frenet IFE method.} 
 
\begin{theorem}\label{thm:coercivity}
    {For $\sigma_0>0$ large enough}, we have 
    \begin{align*}
        a_h(u_h,u_h) \ge \frac{1}{4}\vertiii{u_h}_h^2,\qquad \forall u_h\in \V^m_{\beta}(\T_h). 
    \end{align*}
\end{theorem}

\begin{proof}
    From the definition of $a_h$, we have 

    \begin{equation}a_h(u_h,u_h) = \sum_{K\in T_h} \norm{\sqrt{\beta}\nabla u_h}_{L^2(K)}^2 -2\sum_{e\in \E_h} \left\langle\cc{\beta \nabla u_h\cdot \n}_e,\bb{u_h}_e\right\rangle_{e}  +\frac{\gamma\sigma_0}{h} \sum_{e\in \E_h}\norm{\bb{u_h}_e}_{L^2(e)}^2.
        \label{eqn:diagonal_a_h}
    \end{equation}
Then, by the triangle inequality, Cauchy-Schwarz inequality and Young inequality, we have 
\begin{align}\left|\sum_{e\in \E_h} \left\langle\cc{\beta \nabla u_h\cdot \n}_e,\bb{u_h}_e\right\rangle_{e}\right| &\le \sum_{e\in \E_h} \norm{\cc{\beta \nabla u_h\cdot \n}_e}_{L^2(e)} \norm{\bb{u_h}_e}_{L^2(e)}\notag 
    \\ 
   & \le \frac{h}{\alpha \gamma} \sum_{e\in \E_h} \norm{\cc{\beta \nabla u_h\cdot \n}_e}_{L^2(e)}^2 +\frac{\gamma \alpha}{2h}\sum_{e\in \E_h}\norm{\bb{u_h}_e}_{L^2(e)}^2,
   \label{eqn:flux_term_first_bound}
\end{align}
where $\alpha>0$. Next, we use the trace inequality {in \autoref{thm:trace_inverse_K} }to obtain

\begin{equation}
    \frac{h}{ \gamma}  \sum_{e\in \E_h} \norm{\cc{\beta \nabla u_h\cdot \n}_e}_{L^2(e)}^2 \le 
    \frac{h}{\gamma}\sum_{K\in \T_h} \norm{\beta \nabla u_h\big|_K}_{L^2(\p K)}^2
     \le C_t^2 \sum_{K\in \T_h} \norm{\sqrt{\beta}\nabla u_h}_{L^2(K)}^2.
     \label{eqn:flux_term_second_bound}
\end{equation}
Thus, by substituting \eqref{eqn:flux_term_first_bound} and \eqref{eqn:flux_term_second_bound} into \eqref{eqn:diagonal_a_h}, we find
\begin{equation}
    a_h(u_h,u_h)\ge \left(1-\frac{C_t^2}{\alpha}\right) \sum_{K\in \T_h}\norm{\sqrt{\beta}\nabla u_h}_{L^2(K)}^2 +\frac{\gamma\left(2\sigma_0 -\alpha\right)}{2h}\sum_{e\in \E_h}\norm{\bb{u_h}_e}_{L^2(e)}^2. \notag
\end{equation}
In particular, for $\alpha=2C_t^2$, we have 
\begin{equation}
    a_h(u_h,u_h)\ge \frac{1}{2}\sum_{K\in \T_h}\norm{\sqrt{\beta}\nabla u_h}_{L^2(K)}^2 +\frac{\gamma\left(\sigma_0 -C_t^2\right)}{h}\sum_{e\in \E_h}\norm{\bb{u_h}_e}_{L^2(e)}^2.
    \label{eqn:coercivity_norm_h}
\end{equation}
On the other hand, the inequality \eqref{eqn:flux_term_second_bound} leads to 
\begin{equation}
    \frac{h}{\sigma_0 \gamma}  \sum_{e\in \E_h} \norm{\cc{\beta \nabla u_h\cdot \n}_e}_{L^2(e)}^2 \le \frac{C_t^2}{\sigma_0} a(u_h,u_h).
    \label{eqn:flux_term_third_bound}
\end{equation}
Therefore, by taking $\sigma_0>C_t^2+\frac{1}{2}$ and adding \eqref{eqn:flux_term_third_bound} to \eqref{eqn:coercivity_norm_h}, we obtain  

\begin{equation}
    2 a_h(u_h,u_h)\ge \frac{1}{2}\vertiii{u_h}^2_h,\notag
\end{equation}
leads to the desired result. 
\end{proof}

{Finally, we present the optimal error estimates for the Frenet IFE method in the energy {norm $\vertiii{\cdot}_h$ as well as in the $L^2$ norm} in the following theorem, which are similar to those in Theorems 6.2 and 6.3 in \cite{guoHigherDegreeImmersed2019}, and their proof follows exactly the same arguments 
together with the application of \autoref{lem:proj_energy_error} and 
\autoref{thm:coercivity}.} Consequently, we have opted to omit the proof. 

\begin{theorem}\label{thm:main_error_theorem}
    Let $u\in \H^{m+1}(\O,\Gamma;\beta)$ be a solution to the interface problem \eqref{eqn:Interface_problem}-\eqref{eqn:extended_jump_condition}. Assume that the mesh is fine enough for the results in the previous sections to hold, and that $\sigma_0$ is large enough for \autoref{thm:coercivity} to hold. Then, 
    $$\vertiii{u-u_h}_h\lesssim \frac{\beta^+}{\sqrt{\beta^-}}h^m \norm{u}_{\H^{m+1}(\O)},$$ 
    and 
    $$\norm{u-u_h}_{L^2(\O)}\lesssim \left(\frac{\beta^+}{\beta^-}\right)^2h^{m+1} \norm{u}_{\H^{m+1}(\O)}.$$
\end{theorem}
\begin{proof}See Theorem 6.2 and 6.3 in \cite{guoHigherDegreeImmersed2019}.
\end{proof}

\section{Conclusion}
In this paper, we present an error analysis for the Frenet immersed finite element (IFE) method based on the 
standard SIPDG formulation. This method was previously developed in \cite{adjeridHighOrderGeometry2024} for solving an elliptic interface problem. The Frenet IFE space used in this method is conforming to the space in the weak form of the 
interface problem so that it does not use any penalty on  the interface. The analysis has two technical contributions. The first result is a trace inequality for IFE functions on the Frenet interface element $\hK$ associated with each interface element $K$, where IFE functions are piecewise polynomials and $\hK$ is a quadrilateral with curved sides. Then, on every interface element, a trace inequality is derived for Frenet IFE functions which are not polynomials in general. This trace inequality is critical for establishing the coercivity of the bilinear form in the Frenet IFE method. The optimal error estimates for the Frenet IFE method follow from the standard analysis framework available in the literature. 

%

\printbibliography


%
\end{document}